\documentclass[11pt, twoside, leqno]{article}

\usepackage{amssymb}
\usepackage{amsmath}
\usepackage{amsthm}
\usepackage{color}
\usepackage{mathrsfs}

\allowdisplaybreaks

\def\rr{{\mathbb R}}
\def\rn{{{\rr}^n}}
\def\zz{{\mathbb Z}}

\def\nn{{\mathbb N}}

\def\ca{{\mathcal A}}
\def\cb{{\mathcal B}}

\def\cd{{\mathcal D}}
\def\ce{{\mathcal E}}
\def\cf{{\mathcal F}}

\def\cm{{\mathcal M}}

\def\fz{\infty}
\def\az{\alpha}
\def\bz{\beta}

\def\lz{\lambda}

\def\tz{\theta}

\def\vz{\varphi}

\def\lf{\left}
\def\r{\right}

\def\hs{\hspace{0.25cm}}
\def\ls{\lesssim}
\def\gs{\gtrsim}

\def\noz{\nonumber}
\def\wz{\widetilde}

\def\loc{{\mathop\mathrm{\,loc\,}}}
\def\supp{\mathop\mathrm{\,supp\,}}

\def\rect{\mathop\mathrm{rect\,}}

\def\mcm{\mathcal{M}}

\def\mscX{\mathscr{X}}

\def\modp{\textup{Mod}_p}
\def\Modpq{\textup{Mod}_p^q}
\def\modpq{\Modpq}

\def\mpq{\mcm_p^q(\mscX)}
\def\mpsqs{\cm_{p^\ast}^{q^\ast}(\mscX)}

\def\nmpq{NM_p^q(\mscX)}

\def\prenmpq{\widetilde{NM_p^q}(\mscX)}

\def\nmpqrn{NM_p^q(\rn)}

\def\hmpq{HM_p^q(\mscX)}

\def\hmpsqs{HM_{p^\ast}^{q^\ast}(\mscX)}
\def\hmpqrn{HM_p^q(\rn)}

\def\wmpq{WM^q_p(\rn)}

\def\bint{{\ifinner\rlap{\bf\kern.35em--}
\int\else\rlap{\bf\kern.45em--}\int\fi}\ignorespaces}

\def\bbint{{\ifinner\rlap{\bf\kern.35em--}
\hspace{0.078cm}\int\else\rlap{\bf\kern.45em--}\int\fi}\ignorespaces}

\newtheorem{thm}{Theorem}[section]
\newtheorem{lem}[thm]{Lemma}
\newtheorem{cor}[thm]{Corollary}
\newtheorem{prop}[thm]{Proposition}
\theoremstyle{definition}
\newtheorem{rem}[thm]{Remark}
\newtheorem{defn}[thm]{Definition}
\numberwithin{equation}{section}

\numberwithin{equation}{section}

\begin{document}

\arraycolsep=1pt

\title{\bf\Large Morrey-Sobolev Spaces on Metric Measure Spaces
\footnotetext{\hspace{-0.35cm} 2010 \emph{Mathematics Subject Classification}. Primary 46E35; Secondary 42B25,
42B35, 30L99.
\endgraf {\it Key words and phrases}. Sobolev space, Morrey space, upper gradient,
Haj\l asz gradient, metric measure space, maximal operator
\endgraf Dachun Yang is supported by the National
Natural Science Foundation of China (Grant No. 11171027).
Wen Yuan is supported by the National
Natural Science Foundation  of China (Grant No. 11101038)
and the Alexander von Humboldt Foundation. This project is also partially supported
by the Specialized Research Fund for the Doctoral Program of Higher Education
of China (Grant No. 20120003110003) and the Fundamental Research Funds for Central
Universities of China (Grant No. 2012LYB26).
}}
\author{Yufeng Lu, Dachun Yang\footnote{Corresponding
author.}\ \ and Wen Yuan}
\date{ }
\maketitle

\vspace{-0.5cm}

\begin{center}
\begin{minipage}{13cm}
{\small {\bf Abstract}\quad
In this article, the authors introduce the Newton-Morrey-Sobolev
space on a metric measure space  $(\mathscr{X},d,\mu)$.
The embedding of the Newton-Morrey-Sobolev space
into the H\"{o}lder space is obtained if  $\mathscr{X}$
supports a weak Poincar\'{e} inequality and the measure $\mu$
is doubling and satisfies a lower bounded condition.
Moreover, in the Ahlfors $Q$-regular case, a Rellich-Kondrachov
type embedding theorem is also obtained.
Using the Haj\l asz gradient,
the authors also introduce the Haj\l asz-Morrey-Sobolev spaces, and prove
that the Newton-Morrey-Sobolev space coincides with
the Haj\l asz-Morrey-Sobolev space
when $\mu$ is doubling and $\mathscr{X}$ supports a weak Poincar\'{e} inequality.
In particular, on the Euclidean space ${\mathbb R}^n$, the authors obtain the coincidence
among the Newton-Morrey-Sobolev space, the Haj\l asz-Morrey-Sobolev space
and the classical Morrey-Sobolev space.
Finally, when $(\mathscr{X},d)$ is geometrically doubling and $\mu$
a non-negative Radon measure, the boundedness of some modified (fractional) maximal operators on modified
Morrey spaces is presented; as an application, when $\mu$
is doubling and satisfies some measure decay property,
the authors further obtain the boundedness of some
(fractional) maximal operators on Morrey spaces, Newton-Morrey-Sobolev spaces
and Haj\l asz-Morrey-Sobolev spaces.
}
\end{minipage}
\end{center}

\vspace{0.2cm}

\section{Introduction\label{s1}}

\hskip\parindent In 1996, via introducing the notion of Haj\l asz gradients,
Haj\l asz \cite{ha} obtained an equivalent
characterization of the classical Sobolev space on $\rr^n$, which
becomes an effective way to define Sobolev spaces on metric spaces.
From then on, several different approaches to introduce Sobolev spaces on
metric measure spaces were developed;
see, for example, \cite{km,fhk,Sh,hk00,h03,hu03,y03,ks}.

Throughout the paper, $(\mscX,d,\mu)$ denotes a metric measure space
with a \emph{non-trivial Borel regular measure} $\mu$, which
is finite on bounded sets and positive on open sets.
Let $f$ be a measurable function on $\mscX$. Recall that
a non-negative function $g$ on $\mscX$ is called a
\textit{Haj\l asz gradient of $f$}
if there exists a set $E\subset X$ such that $\mu(E)=0$
and, for all $x,\ y\in \mscX\setminus E$,
\begin{equation*}
|f(x)-f(y)|\leq d(x,y)[g(x)+g(y)].
\end{equation*}
The \textit{Haj\l asz-Sobolev space} $M^{1,p}(\mscX)$ with $p\in[1,\fz]$ is
then defined to be the space of all measurable functions $f\in L^p(\mscX)$ which have
Haj\l asz gradients $g\in L^p(\mscX)$. The \emph{norm} of this space is defined by
$$\|f\|_{M^{1,p}(\mscX)}:= \|f\|_{L^p(\mscX)}+\inf\|g\|_{L^p(\mscX)},$$
where the infimum is taken over all Haj\l asz gradients $g$ of $f$.
It was proved in \cite{ha} that, when $\mscX=\rn$ and $p\in(1,\fz],$
$M^{1,p}(\rn)$ coincides with the classical Sobolev space $W^{1,p}(\rn)$.

Over a decade ago, based on the notions of upper
gradients and weak upper gradients, Shanmugalingam  \cite{Sh,s01}
introduced another type of Sobolev spaces on metric measure spaces,
which are called Newtonian spaces or Newton-Sobolev spaces.
These spaces were also proved to coincide with the Haj\l asz-Sobolev spaces
if $\mscX$ supports some Poincar\'e inequality and the measure is doubling.
Now we recall their definitions.

Recall that we call
$\gamma$ a \emph{curve} if it is a continuous
mapping from an interval into $\mscX$. A curve $\gamma$ is said to be
\emph{rectifiable} if its length is finite. All rectifiable curve can be
arc-length parameterized. Without loss of generality, we may assume that
\emph{all curves appearing in this article are
always treated as arc-length parameterized}.

Let $p\in[1,\infty)$ and $\Gamma$
be a family of non-constant
rectifiable curves on $\mscX$. Recall that the
\emph{admissible class} $F(\Gamma)$ for $\Gamma$ is defined by
\begin{equation}\label{f_gamma}
F(\Gamma):=\lf\{\rho\in[0,\infty]:\ \rho\textup{ is Borel measurable and }
\int_\gamma\rho(s)\,ds\geq1\textup{ for all }\gamma\in\Gamma\r\}.
\end{equation}
If $\Gamma$ contains a constant curve, then $F(\Gamma)=\emptyset$.
The \emph{$p$-modulus} of $\Gamma$ is then defined by
$$\modp(\Gamma):=\inf_{\rho\in F(\Gamma)}\|\rho\|_{L^p(\mscX)}^p,$$
where the infimum is taken over all admissible functions $\rho$ in $F(\Gamma)$.
We let the infimum over the empty set always be infinity.
Let $f$ be a measurable function on $\mscX$. A non-negative function
$g$ is called an \emph{upper gradient} of $f$ if, for any curve
$\gamma\in \Gamma_{\rect}$,
\begin{equation}\label{ineq upper gradient}
|f\circ\gamma(0)-f\circ\gamma(l(\gamma))|\leq\int_\gamma g(s)\,ds,
\end{equation}
where $\Gamma_{\rect}$ is the \emph{class of all non-constant rectifiable
curves in $\mscX$}. Moreover, if the inequality
\eqref{ineq upper gradient} holds for all the curves except for a family of
curves of $p$-modulus zero, then we call $g$ a \emph{$p$-weak upper
gradient} of $f$.
The notion of $p$-weak upper gradient was introduced by Heinonen and Koskela
in \cite{hk}; see also \cite{he} and \cite{Sh,s01}.

For all  $p\in[1,\fz)$, denote by the \emph{symbol} $\wz{N}^{1,p}(\mscX)$
the space of all measurable functions $f\in L^p(\mscX)$ which have
$p$-weak upper gradients $g\in L^p(\mscX)$ and, for all $f\in \wz{N}^{1,p}
(\mscX)$, let
$$
\|f\|_{\wz{N}^{1,p}(\mscX)}:= \|f\|_{L^p(\mscX)}+\inf\|g\|_{L^p(\mscX)},
$$
where the infimum is taken over all $p$-weak upper gradients $g$ of $f$.
The \textit{Newton-Sobolev space} ${N}^{1,p}(\mscX)$ is then defined to be the
quotient space ${N}^{1,p}(\mscX):=\wz{N}^{1,p}(\mscX)/\sim$ with the norm
$\|\cdot\|_{{N}^{1,p}(\mscX)}:=\|\cdot\|_{\wz{N}^{1,p}(\mscX)}$, where
$\sim$ is an \emph{equivalence relation} defined by setting, for all $f_1,f_2\in \wz{N}^{1,p}(\mscX)$,
$f_1\sim f_2$ if $\|f_1-f_2\|_{\wz{N}^{1,p}(\mscX)}=0$. It was proved in
\cite[Theorem 4.9]{Sh}
that the Newton-Sobolev space coincides with the Haj\l asz-Sobolev space
if $(X,\mu)$ supports some Poincar\'e inequality and the measure $\mu$ is doubling.
We refer the reader to \cite{Sh, he, bbs, gks, bb} for more properties about
these spaces.

Recently, there were some attempts to study Newtonian type spaces
in more general settings. Durand-Cartagena in \cite{Du} introduced and
studied the Newtonian space $N^{1,\infty}(\mscX)$ in the limit case
$p=\fz$.
Tuominen \cite{Tu} considered Newtonian type spaces
associated with Orlicz spaces by replacing the Lebesgue norm in the
definition of ${N}^{1,p}(\mscX)$ with Orlicz norms.
Using Lorentz spaces instead of
Lebesgue spaces, Costea and Miranda \cite{CJ} introduced Newtonian type spaces related
to Lorentz spaces. Mal\'{y} \cite{maly1442,maly1448}
studied the Newtonian type spaces  associated with
a general \emph{quasi-Banach function lattice} $X$, namely,
a quasi-Banach function space $X$ satisfying that, if $f\in X$
and $|g|\le |f|$ almost everywhere, then $g\in X$ and $\|g\|_X\le \|f\|_X$.

Let $0<p\le q\le \fz$. Recall that the \emph{Morrey space}
$\mcm_p^q(\mscX)$ (see \cite{morrey})  is defined to be
the space of all measurable functions $f$ on $\mscX$ such that
\begin{equation}\label{a}
\left\|f\right\|_{\mcm_p^q(\mscX)}:=\sup_{B\subset\mscX}
[\mu(B)]^{1/q-1/p}\lf[\int_B|f(x)|^p\,d\mu(x)\r]^{1/p}<\fz,
\end{equation}
where the supremum is taken over all balls in $\mscX$.
In recent years, Morrey spaces and the Morrey versions of many classical
function spaces such as Hardy spaces and
Besov spaces, namely, the spaces defined via replacing
Lebesgue norms by Morrey norms in their norms,
attract more and more attentions and have proved useful
in the study of partial differential equations and harmonic analysis; see, for
example, \cite{ad04,ad11,ad12,ad,n06,mnos10,
koya94,m03,ysy10} and their references.

The main purpose of this article is to develop a
theory of  Newtonian type spaces based on Morrey spaces, namely,
Newton-Morrey-Sobolev spaces, as well as
the Haj\l asz-Morrey-Sobolev spaces on metric measure spaces.

We begin with the following generalized modulus based on Morrey spaces.

\begin{defn}\label{def mod}
Let $1\le p\le q<\fz$ and
$\Gamma$ be a collection of rectifiable curves.
The \textit{Morrey-modulus} of $\Gamma$
is defined by
\begin{equation*}
\Modpq(\Gamma):=\inf_{\rho\in F(\Gamma)}\|\rho\|_{\mpq}^p,
\end{equation*}
where $F(\Gamma)$ is defined as in \eqref{f_gamma}.
\end{defn}

\begin{defn}
Let $f$ be a measurable function and $g$ a non-negative
Borel measurable function.
If the inequality
\eqref{ineq upper gradient} holds true for all non-constant rectifiable
curves in $\mscX$ except a family of
curves of Morrey-modulus zero, then
$g$ is called a \textit{$\Modpq$-weak upper gradient} of $f$.
\end{defn}

Via these $\Modpq$-weak upper gradients, the
Newton-Morrey-Sobolev space is introduced as follows.

\begin{defn}\label{def nm space}
Let $1\le p\le q<\fz$. The \emph{space}
$\widetilde{NM_p^q}(\mscX)$ is defined to be
the set of all $\mu$-measurable functions $f$ such that $\|f\|_{\prenmpq}<\infty$,
where
\begin{equation*}
\|f\|_{\widetilde{NM_p^q}(\mscX)}:=\|f\|_{\mpq}+\inf\|g\|_{\mpq}
\end{equation*}
with the infimum being taken over all $\Modpq$-weak upper gradients $g$
of $f$.
The \emph{Newton-Morrey-Sobolev space} $\nmpq$ is then
defined as the quotient space
$$\prenmpq\big/\left\{f\in \widetilde{NM_p^q}(\mscX):\
\|f\|_{\widetilde{NM_p^q}(\mscX)}=0\right\}$$ with
\begin{equation*}
\|f\|_{\nmpq}:=\|f\|_{\prenmpq}.
\end{equation*}
\end{defn}

It is easy to see that $\|\cdot\|_{\nmpq}$ is a norm. Moreover, when $p=q$,
the space $\nmpq$ is just the Newton-Sobolev space $N^{1,p}(\mscX)$
introduced by Shanmugalingam \cite{Sh}.
We also remark that, since Morrey spaces are Banach function lattices, these
Newton-Morrey-Sobolev spaces are special cases of the
Newtonian type spaces associated with
quasi-Banach function lattices considered by Mal\'{y}
\cite{maly1442,maly1448}.

This article is organized as follows. In Section
\ref{s2}, we show that the Newton-Morrey-Sobolev space is non-trivial by proving
that the set of Lipschitz functions with bounded support is
contained in the Newton-Morrey-Sobolev space $\nmpq$ (see Theorem \ref{ne} below),
but not dense in some examples (see Remark \ref{not dense} below),
which is different from the Newton-Sobolev space. Moreover, in Remark \ref{rd}
below, we even show that the set of Lipschitz functions is not dense in
$NM_p^q(\rn)$ when $1<p<q<\fz$.

In Section \ref{s3}, the embedding of the Newton-Morrey-Sobolev space
into the H\"{o}lder space is obtained
when $\mscX$ supports a weak
Poincar\'{e} inequality, the measure $\mu$ is doubling and satisfies
a lower bounded condition (see Theorem \ref{t-emd} below).
Moreover, if the space
$\mscX$ is Ahlfors $Q$-regular and supports a weak
Poincar\'{e} inequality,
via proving the boundedness of some
fractional integrals on Morrey spaces,
we also obtain a Rellich-Kondrachov
type embedding theorem of the Newton-Morrey-Sobolev space
(see Theorem \ref{thm-rk-embd} below). Both embedding properties
on Newton-Morrey-Sobolev spaces generalize the corresponding results for
Newton-Sobolev spaces obtained by Shanmugalingam in \cite[Theorems 5.1 and 5.2]{Sh}.

In Section \ref{s4}, using the
Haj\l asz gradient, we introduce the Haj\l asz-Morrey-Sobolev space
on metric measure spaces and show that,
when $\mscX$ supports a weak Poincar\'{e} inequality and the measure $\mu$
is doubling,
the Newton-Morrey-Sobolev space coincides with
the Haj\l asz-Morrey-Sobolev space (see Theorem \ref{thm-h=n} below).
This generalizes the result on the relation between
Newton-Sobolev spaces and Haj\l asz-Sobolev spaces obtained by Shanmugalingam
in \cite[Theorem 4.9]{Sh}. In particular, when $\mscX=\rn$ and $1<p\le q<\fz$,
both the Newton-Morrey-Sobolev space $\nmpqrn$ and the
Haj\l asz-Morrey-Sobolev space $\hmpqrn$
are proved to coincide with the classical
Morrey-Sobolev space on $\rn$ (see Theorem \ref{t3} below).

Finally, Section \ref{s6} is devoted to the boundedness of some fractional maximal operators
on Morrey and Morrey-Sobolev spaces. We first show, in Subsection \ref{s6.1},
the boundedness of some modified maximal operators
on modified Morrey spaces over geometrically doubling metric measure spaces
(see Theorem \ref{t6.5} below). As an application,
the boundedness of related fractional maximal operators on modified Morrey spaces is obtained
(see Proposition \ref{pf} below). As further applications,
in Subsection \ref{s6.2}, we show the boundedness of
(fractional) maximal operators on Haj\l asz-Morrey-Sobolev spaces when $\mscX$
is a doubling metric measure space satisfying the  relative
$1$-annular decay property and the measure  lower bound condition
(see Theorem \ref{tf} below). If $\mscX$ supports a weak Poincar\'{e}-inequality,
and the measure is doubling and satisfies the
measure  lower bound condition, then the boundedness of discrete
(fractional) maximal operators on Newton-Morrey-Sobolev spaces is also
obtained (see Theorem \ref{thm-frac-max-bd1} below).
All these conclusions generalize the corresponding known results on Newton-Sobolev spaces
and Haj\l asz-Sobolev spaces by Heikkinen et al. in \cite{hknt,hlnt}.

At the end of this section, we make some conventions on notation.
Throughout the paper, we denote by $C$ a \emph{positive constant}
which is independent of the main parameters, but it may vary from
line to line. The \emph{symbols} $A\lesssim B$ and $A\gs B$
means $A\leq CB$ and $A\ge CB$, respectively,
where $C$ is a positive constant. If $A\lesssim B$ and
$B\lesssim A$, then we write $A\approx B$. If $E$ is a subset of
$\mscX$, we denote by $\chi_E$ its \emph{characteristic function}.

\section{Some basic properties\label{s2}}

\hskip\parindent In this section, we consider some basic properties of
Newton-Morrey-Sobolev spaces including their completeness and non-triviality.
Throughout this section, we \emph{only assume that $\mu$ is a non-trivial Borel
regular measure}.

Recall that the Newton-Morrey-Sobolev space is a special case of
the Newtonian spaces based on quasi-Banach function lattice $X$
introduced in \cite{maly1442}. The following result is a special case
of  \cite[Theorem 7.1]{maly1442}.

\begin{thm}[]\label{nmpq complete}
For all  $1\le p\le q<\fz$, the space $\nmpq$ is a Banach space.
\end{thm}

The next lemma is usually called the \emph{truncation lemma}, which
shows how a $\Modpq$-weak upper gradient behaves when multiplying a
characteristic function. Its proof is similar to those of
\cite[Lemmas 4.6 and 4.7]{CJ}, the details being omitted.

\begin{lem}\label{truncation lemma}
Let $f\in\nmpq$ and $g_1,g_2\in \mpq$ be two
$\modpq$-weak upper gradients of $f$.

{\rm(i)} If $f$ is a constant on a closed set $E$, then $g:=g_1
\chi_{\mscX\setminus E}$ is also a $\modpq$-weak upper gradient of $f$.

{\rm(ii)} If $E$ is closed in $\mscX$, then
$$h:=g_1\chi_E+g_2\chi_{\mscX\setminus E}$$
is also a $\modpq$-weak upper gradient of $f$.
\end{lem}

We also need the following conclusion.

\begin{prop}\label{mor}
Let $1\le p\le q<\fz$. For any set $E\subset\mscX$ with
finite measure, $\|\chi_E\|_{\mpq}$ is bounded by a positive
constant multiple of $[\mu(E)]^{1/q}$ with the positive
constant independent of $E$.
\end{prop}

\begin{proof}
Notice that
\begin{eqnarray*}
\|\chi_{E}\|_{\mpq}=\sup_{B\subset \mscX} [\mu(B)]^{1/q}
\lf[\frac{\mu(B\cap E)}{\mu(B)}\r]^{1/p}.
\end{eqnarray*}
If $\mu(B)\ge \mu(E)/2$, then by $p\le q$, we have
$$[\mu(B)]^{1/q}
\lf[\frac{\mu(B\cap E)}{\mu(B)}\r]^{1/p}\ls [\mu(E)]^{1/q-1/p}
[\mu(B\cap E)]^{1/p}\ls [\mu(E)]^{1/q}.$$
If $\mu(B)\le \mu(E)/2$, then
$$[\mu(B)]^{1/q}\lf[\frac{\mu(B\cap E)}{\mu(B)}\r]^{1/p}\ls [\mu(E)]^{1/q}.$$
This finishes the proof of Proposition \ref{mor}.
\end{proof}

Recall that $NM_p^p(\mscX)=N^{1,p}(\mscX)$, which is a non-trivial space,
namely, the space $N^{1,p}(\mscX)$ contains
more than just the zero function and might be a proper subspace of $L^p(\mscX)$
if $\mscX$ has enough rectifiable paths (see \cite{Sh}).
The following conclusion shows that, even when $q>p\geq1$, $NM_p^q(\mscX)$ is also
a non-trivial space. In what follows, $\textup{Lip}_b(\mscX)$
denotes the \emph{set of all Lipschitz functions on $\mscX$ with bounded support}.

\begin{thm}\label{ne}
Let $1\le p\le q<\fz$. Then,
$$\textup{Lip}_b(\mscX)\subset\nmpq \subset {N}_{\rm loc}^{1,p}(\mscX),$$
where ${N}_{\rm loc}^{1,p}(\mscX)$
denotes the collection of functions which belong to $N^{1,p}(B)$ for any
ball $B\subset\mscX$.
\end{thm}

\begin{proof}
To show the first embedding, let $B$ be a ball in $\mscX$. By
Proposition \ref{mor}, we know that $\chi_B\in \mpq$
and
$$\|\chi_B\|_{\mpq}\ls [\mu(B)]^{1/q}<\fz.$$
Recall, by our conventions on notation at the end of Section \ref{s1},
that the symbol $\ls$ means that the implicit positive constant
here is independent of $B$.

Now let $f\in \textup{Lip}_b(\mscX)$ with
$\supp(f)\subset B$ and $L$ be the Lipschitz constant of $f$,
which means that, for all $x,y\in\mscX$,
$$|f(x)-f(y)|\le Ld(x,y).$$
Since $f\in \textup{Lip}_b(\mscX)$, we know that there exists a positive
constant $M_0$ such that $|f|\le M_0\chi_B$. Hence, by \eqref{a},
we see that $f\in \mpq$ and
$$\|f\|_{\mpq}\ls M_0[\mu(B)]^{1/q}.$$

On the other hand, notice that, for all rectifiable curves $\gamma$, it holds true that
$$|f\circ \gamma(\ell(0))-f\circ\gamma(\ell(\gamma))|\le L\,d(\gamma(\ell(0)),
\gamma(\ell(\gamma)))\le \int_\gamma L \,ds.$$
Hence $L$ is an upper gradient of $f$. Then, by Lemma \ref{truncation lemma},
$L\chi_{2B}$ is a $\Modpq$-weak upper gradient of $f$, which further implies that
$f\in\nmpq$ and
$$\|f\|_{\nmpq}\ls (M_0+L)[\mu(B)]^{1/q}.$$
Thus, $\textup{Lip}_b(\mscX)\subset\nmpq$.

The second embedding follows directly from definitions, together with
\cite[Corollary 5.7]{maly1442}. Indeed, let $f\in\nmpq$. Then, by
\cite[Definition 2.4 and Corollary 5.7]{maly1442},
we know that
$$\|f\|_{\nmpq}=\|f\|_{\mpq}+\inf\|h\|_{\mpq}<\fz,$$
where the infimum is taken over all the upper gradients of $f$. From this,
we deduce that $f$ has an upper gradient $h\in\mpq$ and hence $h\in L^p(E)$
for any ball $E\subset\mscX$. Since it is obvious that
$f\in L^p(E)$, by \cite[Definition 2.4 and Corollary 5.7]{maly1442}
again, we obtain that $f\in N^{1,p}(E)$, which, together with
the arbitrariness of $E\subset\mscX$ and the definition of $N^{1,p}_{\loc}(\mscX)$,
implies that $f\in N^{1,p}_{\loc}(\mscX)$ and hence completes the proof of Theorem \ref{ne}.
\end{proof}

\begin{rem}\label{not dense}
We point out that $\textup{Lip}_b(\mscX)$ might not be dense in $\nmpq$ when $p<q$.
Indeed, even the \emph{set of Lipschitz
functions, $\textup{Lip}(\mscX)$, might not be dense in $\nmpq$ when $p<q$}.
This behavior of $\nmpq$ (non-density of Lipschitz functions) is
different from the Newton-Sobolev space ${N}^{1,p}(\mscX)=NM^p_p(\mscX)$,
since $\textup{Lip}(\mscX)$  is dense in ${N}^{1,p}(\mscX)$
(see \cite[Theorem 4.1]{Sh}).
A counterexample in the Euclidean setting is given in Remark \ref{rd} below.
\end{rem}

\section{Sobolev embeddings\label{s3}}

\hskip\parindent Let $\az\in(0,1]$ and $C^{0,\az}(\mscX)$ denote the
\emph{$\az$-H\"older space} on $\mscX$,
namely, the space of all functions $f$ satisfying that, for all $x,y\in\mscX$,
$$|f(x)-f(y)|\le C [d(x,y)]^\az,$$
where $C$ is a positive constant independent of $x$ and $y$.

It is well known that, when $\mscX=\rn$, the following Sobolev embeddings
hold true:
\begin{eqnarray}\label{ebd1}
W^{1,p}(\rn)&&\hookrightarrow L^{np/(n-p)}(\rn)\quad\textup{if } p<n,
\end{eqnarray}
and
\begin{eqnarray}\label{ebd2}
W^{1,p}(\rn)&&\hookrightarrow C^{0,1-n/p}(\rn)\quad\textup{if } p>n,
\end{eqnarray}
where the \emph{symbol} $\hookrightarrow$ means continuous embedding.
The generalizations of \eqref{ebd1} and \eqref{ebd2} to
the Newton-Sobolev space and the Haj{\l}asz-Sobolev space
on metric measure spaces were obtained in \cite{Sh} and \cite{hk,hk00}, respectively.
This section is devoted to the corresponding Sobolev embedding
theorems for Newton-Morrey-Sobolev spaces.

Recall that a space $\mscX$ is said to
\emph{support a weak $(1,p)$-Poincar\'e inequality}
if there exist positive constants $C$ and $\tau\ge1$ such
that, for all open balls $B$ in $\mscX$ and all
pairs of functions $f$ and $\rho$ defined on $\tau B$,
whenever $\rho$ is an upper gradient of
$f$ in $\tau B$ and $f$ is integrable on $B$,
then
\begin{equation}\label{pi}
\frac{1}{\mu(B)}\int_B|f(x)-f_B|\,d\mu(x)\le C\, {\rm
diam}(B)\lf\{\frac{1}{\mu(\tau B)}\int_{\tau B} \lf[\rho(x)
\r]^p\,d\mu(x)\r\}^{1/p},
\end{equation}
where above and in what follows,
$f_B$ denotes the \emph{integral mean of $f$ on $B$},
namely,
\begin{equation}\label{eq-def-mean}
f_B=\frac{1}{\mu(B)}\int_Bf(y)\,d\mu(y),
\end{equation}
$\textup{diam}(B)$ the \textit{diameter} of $B$ and
$\tau B$ the \textit{ball with the same center as $B$
but $\tau$ times the radius of $B$}.
In particular, if $\tau=1$, then we say that $\mscX$
\emph{supports a $(1,p)$-Poincar\'e inequality}.

It is well known that the Euclidean space supports
a $(1,p)$-Poincar\'e inequality.
For more information on Poincar\'e inequalities, we refer the reader
to \cite{hek1,hek2,hk00} and their references.

A measure $\mu$ on $\mscX$ is said to be
\emph{doubling} if there exists a positive constant
$C$ such that, for all balls $B$ in $\mscX$, it holds true that
$\mu(2B)\le C\mu(B)$.
As a generalization of \eqref{ebd2} to  Newton-Morrey-Sobolev spaces, we have
the following conclusion.

\begin{thm}\label{t-emd}
Let $1\leq p\le q<\fz$ and $Q\in(0,q)$. Assume that
$(\mscX, d,\mu)$ is a metric measure space, with doubling measure $\mu$,
and supports a weak $(1,p)$-Poincar\'{e} inequality.
If there exists a positive constant $C$ such that $\mu(B(x,r))\geq Cr^Q$
for all $x\in\mscX$ and $0<r<2
\textup{diam}(\mscX)$, then
$$\nmpq \hookrightarrow C^{0,1-Q/q}(\mscX).$$
\end{thm}

\begin{proof}
By the same reason as that stated in the proof of \cite[Theorem 5.1]{Sh},
we only need to show that, if $f\in\nmpq$ and $x,y$ are
Lebesgue points of $f$, then
\begin{equation*}
|f(x)-f(y)|\ls[d(x,y)]^{1-Q/q}\|f\|_{\nmpq}.
\end{equation*}

To this end, let $B_1:=B(x,d(x,y)),$ $B_{-1}:=B(y,d(x,y))$ and,
 for all $i>1$,
\begin{equation*}
B_i=\frac{1}{2}B_{i-1}\textup{ and }B_{-i}=\frac{1}{2}B_{-i+1}.
\end{equation*}
Let $B_0:=B(x,2d(x,y))$.
Since $x,y$ are Lebesgue points, it follows that
\begin{equation*}
|f(x)-f(y)|\leq\sum_{i\in\zz}|f_{B_i}-f_{B_{i+1}}|.
\end{equation*}

Let $\rho$ be an upper gradient of $f$ such that
$$\|f\|_{\mpq}+\|\rho\|_{\mpq}\ls \|f\|_{\nmpq}.$$
Let $r_i$ be the radius of the ball $B_i$. Then, by this, \eqref{a}, the doubling condition of $\mu$ and
the weak $(1,p)$-Poincar\'e inequality, together with $\mu(\tau B_i)\gtrsim r_i^Q$,
we see that, when $i\in\nn$,
\begin{eqnarray}\label{3.4}
|f_{B_i}-f_{B_{i+1}}|
&&\ls \frac1{\mu(B_i)}\int_{B_i}|f_{B_i}-f(x)|\,d\mu(x)\\
&&\lesssim \textup{diam}(B_i)\left\{\frac1{\mu(\tau B_i)}\int_{\tau B_i}[\rho(x)]^p
\,d\mu(x)\right\}^{1/p}\noz\\
&&\lesssim r_i[\mu(\tau B_i)]^{-1/q}
\|\rho\|_{\mpq}\lesssim r_i^{1-Q/q}\|\rho\|_{\mpq}\noz\\
&&\lesssim 2^{-i(1-Q/q)}[d(x,y)]^{1-Q/q}\|f\|_{\nmpq}.\noz
\end{eqnarray}
Similarly, for all $i\le -2$, we also have
\begin{eqnarray*}
|f_{B_i}-f_{B_{i+1}}|
\ls 2^{i(1-Q/q)} [d(x,y)]^{1-Q/q}\|f\|_{\nmpq}.
\end{eqnarray*}

On the other hand, by the H\"older inequality and the doubling condition
of $\mu$, we see that
\begin{eqnarray*}
|f_{B_{-1}}-f_{B_0}|
&&\le\frac{1}{\mu(B_{-1})}\int_{B_{-1}}\lf|f_{B_{0}}-f(z)\r|\,d\mu(z)
\ls\frac{1}{\mu(B_0)}\int_{B_{0}}|f_{B_{0}}-f(x)|\,d\mu(x)
\end{eqnarray*}
and then, similar to \eqref{3.4}, we further conclude that
$$|f_{B_{-1}}-f_{B_0}|\ls[d(x,y)]^{1-Q/q} \|f\|_{\nmpq}.$$
Meanwhile, by the same method as above, we also find that
$$|f_{B_{0}}-f_{B_1}|\ls[d(x,y)]^{1-Q/q} \|f\|_{\nmpq}.$$

Thus, combining the above estimates, by $Q\in (0,q)$, we see that
\begin{eqnarray*}
|f(x)-f(y)|&&\lesssim [d(x,y)]^{1-Q/q}\left[\sum_{i\in\zz}2^{-|i|
(1-Q/q)}\right]\|f\|_{\nmpq}\\
&&\ls [d(x,y)]^{1-Q/q} \|f\|_{\nmpq},
\end{eqnarray*}
which completes the proof of Theorem \ref{t-emd}.
\end{proof}

\begin{rem}
Theorem \ref{t-emd} generalizes \cite[Theorem 5.1]{Sh} by taking $p=q$.
\end{rem}

Next we give a Rellich-Kondrachov type embedding theorem for $\nmpq$ when
$p$ is small, which can be seen as a generalization of \eqref{ebd1}.
We begin with the following notion of the Ahlfors $Q$-regular measure spaces;
see, for example, \cite{he}.

\begin{defn}\label{def-Qregular}
Let $Q\in(0,\fz)$. A metric measure space $\mscX$ is
said to be \emph{Ahlfors $Q$-regular}
(or \emph{$Q$-regular}), if there exists a
constant $C\geq1$ such that, for any $x\in\mscX$
and any $r\in(0,2{\rm diam}(\mscX)),$
\begin{equation*}
\frac{1}{C}\,r^Q\leq\mu(B(x,r))\leq Cr^Q.
\end{equation*}
\end{defn}

Let $L_\loc^1(\mscX)$ be the collection of all locally integrable
functions on $\mscX$.
The \textit{Hardy-Littlewood maximal operator} $M$ is defined by setting,
for all $f\in L^1_\loc(\mscX)$ and $x\in\mscX$,
\begin{equation}\label{def max}
Mf(x):=\sup_{B\ni x}\frac1{\mu(B)}\int_B|f(y)|\,d\mu(y),
\end{equation}
where the supremum is taken over all balls $B$ in $\mscX$ containing $x$.
The following statement shows that the operator $M$ is bounded on Morrey spaces.
For its proof, we refer the reader to \cite{am} for example.

\begin{lem}\label{lem_bd of hl in morrey}
Let $(\mscX,d,\mu)$ be a metric space with doubling measure $\mu$ and
$1<p\leq q\le\infty$. Then there exists a positive constant $C$ such
that, for all $f\in\mpq$,
\begin{equation*}
\|Mf\|_{\mpq}\leq C\|f\|_{\mpq}.
\end{equation*}
\end{lem}

We also need the following
boundedness of fractional integral operators on Morrey spaces.

\begin{prop}\label{lem-bd-of-int-op-morrey}
Let $\mscX$ be Ahlfors $Q$-regular with $Q\in(0,\fz)$,
$1<p\le q<\fz$ and $\az>0$ such that
$q<Q/\alpha$. Then, the fractional integral $I_\alpha$ is
bounded from $\mcm_p^q(\mscX)$ to $\mcm_{p^\ast}^{q^\ast}(\mscX)$, where
$p^\ast:=\frac{Qp}{Q-q\alpha}$,
$q^\ast:= \frac{Qq}{Q-q\alpha}$ and $I_\alpha$
is defined by setting, for all $f\in\mpq$ and $x\in\mscX$,
\begin{equation*}
I_\alpha(f)(x):=\int_\mscX\frac{f(y)}{[d(x,y)]^{Q-\alpha}}\,d\mu(y).
\end{equation*}
\end{prop}

\begin{proof}
Without loss of generality, we may assume that $f\in \mpq$ is
non-negative. For any $x\in\mscX$, fix $\delta>0$ and write
\begin{eqnarray*}
I_\alpha(f)(x)&=& \int_{B(x,\delta)}\frac{f(y)}{[d(x,y)]^{Q-\alpha}}\,d\mu(y)+
\int_{\mscX\setminus B(x,\delta)}\frac{f(y)}{[d(x,y)]^{Q-\alpha}}\,d\mu(y)\\
&=:&b^{(\alpha)}_\delta(x)+g^{(\alpha)}_\delta(x).\nonumber
\end{eqnarray*}

By the H\"{o}lder inequality, \eqref{a} and the Ahlfors $Q$-regular property of
$\mscX$, together with $q<Q/\az$, we see that
\begin{eqnarray}\label{eqn-estimate-for-good}
g^{(\alpha)}_\delta(x)&&=\sum_{j=0}^\fz \int_{B(x,2^{j+1}\delta)
\setminus B(x,2^j\delta)} \frac{f(y)}{[d(x,y)]^{Q-\alpha}}\,d\mu(y)\\
&&\ls \sum_{j=0}^\fz \lf(2^j\delta\r)^{\alpha-Q} \lf[\mu\lf(B\lf(x,2^{j+1}\delta\r)\r)\r]^{1-1/p}
\lf\{\int_{B(x,2^{j+1}\delta)}[f(y)]^p
\,d\mu(y)\r\}^{1/p}\nonumber\\
&&\ls \sum_{j=0}^\fz \lf(2^j\delta\r)^{\alpha-Q} \lf[\mu\lf(B\lf(x,2^{j+1}\delta\r)\r)\r]^{1-1/q}
\|f\|_{\mpq}\nonumber\\
&&\approx \sum_{j=0}^\fz \lf(2^j\delta\r)^{\alpha-Q} \lf(2^j\delta\r)^{(1-1/q)Q} \|f\|_{\mpq}
\ls \delta^{\alpha-Q/q}\|f\|_{\mpq}.\nonumber
\end{eqnarray}

For $b_\delta$, let $A_j:=B_j\setminus B_{j+1}:=B(x,2^{-j}\delta)\setminus
B(x,2^{-j-1}\delta)$ for all $j\in\nn\cup\{0\}=:\zz_+$. Then, by the
Ahlfors $Q$-regular property of $\mscX$, together with $\az>0$, we see that, for all $x\in\mscX$,
\begin{eqnarray}\label{eqn-estimate-for-bad}
b^{(\alpha)}_\delta(x)&=&\sum_{j\in\zz_+}\int_{A_j}\frac{f(y)}
{[d(x,y)]^{Q-\alpha}}\,d\mu(y)
\approx\sum_{j\in\zz_+}\lf(2^{-j}\delta\r)^{\alpha-Q}\int_{B_j}f(y)\,d\mu(y)\\
&\ls& \delta^\alpha\sum_{j\in\zz_+}2^{-j\alpha}\frac{1}{\mu(B_j)}
\int_{B_j}f(y)\,d\mu(y)\ls\delta^\alpha M(f)(x).\nonumber
\end{eqnarray}

Combining
\eqref{eqn-estimate-for-good} and \eqref{eqn-estimate-for-bad}, we have
\begin{equation*}
I_\alpha(f)(x)\ls \delta^\alpha M(f)(x)+ \delta^{\alpha-Q/q}\|f\|_{\mpq}.
\end{equation*}
Now let $\delta:=\|f\|_{\mpq}^{q/Q}[ M(f)(x)]^{-q/Q}$.
Then for any $x\in\mscX$,
\begin{equation*}
I_\alpha(f)(x)\ls \|f\|_{\mpq}^{\alpha q/Q}[M(f)(x)]^{1-\alpha q/Q},
\end{equation*}
which, together with Lemma \ref{lem_bd of hl in morrey}, further implies that
\begin{eqnarray*}
\|I_\alpha(f)\|_{\mcm_{p^\ast}^{q^\ast}(\mscX)}&\ls&\|f\|_{\mpq}^{q\alpha/Q}
\lf\|[M(f)]^{1-q\alpha/Q}\r\|_{\mcm_{p^\ast}^{q^\ast}(\mscX)}\\
&\approx &\|f\|_{\mpq}^{q\alpha/Q}\|M(f)\|_{\mpq}^{1-q\alpha/Q}
\ls \|f\|_{\mpq}.
\end{eqnarray*}
This finishes the proof of Proposition \ref{lem-bd-of-int-op-morrey}.
\end{proof}

Now we have the following Rellich-Kondrachov type embedding
result, which generalizes \cite[Theorem 5.2]{Sh} by taking
$p=q$, $\alpha=1$ and $\mscX$ being bounded.

\begin{thm}\label{thm-rk-embd}
Let $1\le r<p\leq q<\infty$. Let $1<p/r\le q/r<Q/\az<\fz$,
$\alpha\in(0,r)\cap(0,Q)$ and $\mscX$ be an Ahlfors
$Q$-regular metric measure space supporting a
weak $(1,r)$-Poincar\'{e} inequality.
Then there exists a positive constant $C$ such that,
for all functions $f\in\nmpq$, upper gradients $\rho$ of $f$ and
$R\in(0,\fz)$,
\begin{equation*}
\|f-f_{B(\cdot,R)}\|_{\mcm_{Q^*p}^{
Q^*q}(\mscX)}\leq C R^{1-\alpha/r}\|\rho\|_{\mpq},
\end{equation*}
where $Q^*:=\frac{Qr}{Qr-q\az}$.
\end{thm}

\begin{proof}
Let $f\in\nmpq$ and $\rho$ be an
upper gradient of $f$. For any Lebesgue point $x$ for $f$,
we write $B_0:=B(x,R)$ and
$B_i:=B(x,2^{-i}R)$ for all $i\in\nn$.
Since an Ahlfors $Q$-regular space is doubling, by the weak $(1,r)$-Poincar\'{e}
inequality (namely, the inequality \eqref{pi}
with $p$ replaced by $r$) and $r>\az$, we see that
\begin{eqnarray*}
|f(x)-f_{B(x,R)}|&\leq&\sum_{i=0}^\infty|f_{B_i}-f_{B_{i+1}}|
\ls\sum_{i=0}^\infty\frac{1}{\mu(B_i)}\int_{B_i}|f(z)-f_{B_i}|\,d\mu(z)\\
&\ls&\sum_{i=0}^\infty\frac{\textup{diam}(B_i)}{[\mu(\tau B_i)]^{1/r}}
\lf\{\int_{\tau B_i}[\rho(z)]^r\,d\mu(z)\r\}^{1/r}\nonumber\\
&\approx&\sum_{i=0}^\infty\frac{\textup{diam}(B_i)}{(2^{-i}\tau R)^{Q/r}}\lf\{
\int_{\tau B_i}[\rho(z)]^r\,d\mu(z)\r\}^{1/r}\nonumber\\
&\ls&\sum_{i=0}^\infty\frac{2^{-i}R}{(2^{-i}\tau R)^{\alpha/r}}\lf\{
\int_{\tau B_i}\frac{[\rho(z)]^r}{[d(x,z)]^{Q-\alpha}}\,d\mu(z)
\r\}^{1/r}\nonumber\\
&\ls& R^{1-\alpha/r}\lf\{\int_{\mscX}\frac{[\rho(z)]^r}
{[d(x,z)]^{Q-\alpha}}\,d\mu(z)\r\}^{1/r}\approx
R^{1-\alpha/r}[I_\alpha(\rho^r)(x)]^{1/r}.
\end{eqnarray*}
Applying Proposition \ref{lem-bd-of-int-op-morrey}, together with
$1<p/r\le q/r<Q/\az$, we conclude that
\begin{eqnarray}
\|f-f_{B(\cdot,R)}\|_{\mcm_{Q^*p}^{
Q^*q}(\mscX)}
&\ls& R^{1-\alpha/r}\lf\|[I_\alpha(\rho^r)]^{1/r}\r\|_{
\mcm_{Q^*p}^{Q^*q}(\mscX)}
\approx R^{1-\alpha/r}\lf\|I_\alpha(\rho^r)\r\|^{1/r}_{
\mcm_{Q^*p/r}^{Q^*q/r}(\mscX)}\nonumber\\
&\ls& R^{1-\alpha/r}\|\rho^r\|^{1/r}_{
\mcm_{p/r}^{q/r}(\mscX)}\approx R^{1-\alpha/r}\|\rho\|_{
\mcm_p^q(\mscX)}\nonumber,
\end{eqnarray}
which completes the proof of Theorem \ref{thm-rk-embd}.
\end{proof}

\begin{rem}
(i) Let $1<r<p<\fz$, $1<p/r<Q<\fz$,
and $\mscX$ be an Ahlfors $Q$-regular metric measure space supporting a
weak $(1,r)$-Poincar\'{e} inequality. Then, by Theorem \ref{thm-rk-embd}
with $\alpha=1$,
we see that there exists a positive constant $C$ such that, for all functions
$f\in N^{1,p}(\mscX)$, upper gradients $\rho$ of $f$ and $R\in(0,\fz)$,
\begin{equation*}
\|f-f_{B(\cdot,R)}\|_{L^{\frac{Qpr}{Qr-p}}(\mscX)}\leq CR^{1-1/r}\|\rho\|_{L^p(\mscX)},
\end{equation*}
which has its own interest. However, it is not clear whether the above conclusion still holds
true for the case $r=1$ or not, since, we had to use Theorem \ref{thm-rk-embd}
with $\az=1$ and, to this end, we need $r>\alpha=1$.

(ii) We also remark that  Theorem \ref{thm-rk-embd} generalizes the classical
result for Newton-Sobolev spaces in \cite[Theorem 5.2]{Sh}. Indeed,
if we further assume that $\mscX$ is bounded, then we know
that  $f_\mscX=f_{B(x, \mathop{\mathrm{\,diam}}(\mscX))}$ for almost all $x\in\mscX$.
Thus, it follows, from (i), that, under the same assumptions
on $Q,\, r,\, p$ as in (i), there exists a positive constant $C$ such that, for
all functions $f\in N^{1,p}(\mscX)$ and upper gradients $\rho$ of $f$,
\begin{equation*}
\|f-f_{\mscX}\|_{L^{\frac{Qpr}{Qr-p}}(\mscX)}\leq C[{\rm diam}(\mscX)]^{
1-1/r}\|\rho\|_{L^p(\mscX)},
\end{equation*}
which is just \cite[Theorem 5.2]{Sh}.

(iii) The condition on the weak $(1,r)$-Poincar\'e inequality in Theorem \ref{thm-rk-embd}
can be replaced by the weak $(1,1)$-Poincar\'e inequality, due to the H\"older inequality.
\end{rem}

\section{Haj{\l}asz-Morrey-Sobolev spaces\label{s4}}

\hskip\parindent In this section, we  introduce Morrey-Sobolev spaces associated with
Haj\l asz gradients and consider the relation between the
Haj{\l}asz-Morrey-Sobolev space and the Newton-Morrey-Sobolev space.

\begin{defn}\label{def hmpu}
Let $0<p\leq q\leq\fz$. The \emph{Haj\l asz-Morrey-Sobolev space} $\hmpq$ is defined to
be the space of all measurable functions $f$ that have a Haj\l asz
gradient $h\in\mpq$. The \emph{norm} of $f\in\hmpq$ is defined as
\begin{equation*}
\|f\|_{HM_p^q(\mscX)}:=\|f\|_{\mpq}+\inf\|h\|_{\mpq},
\end{equation*}
where the infimum is taken over all Haj\l asz gradients $h$ of $f$.
\end{defn}

We remark that $\hmpq$ when $p=q$ is just the
Haj\l asz-Sobolev space $M^{1,p}(\mscX)$ of \cite{ha}.
Moreover, $\|f\|_{\hmpq}=0$
if and only if $f=0$ almost everywhere.

To consider the relation between the Haj{\l}asz-Morrey-Sobolev
space and the Newton-Morrey-Sobolev space, we need the following technical lemma,
which is a special case of \cite[Lemma 5.6]{maly1442}.

\begin{lem}\label{l-wug2}
Let $1\le p\le q<\fz$ and $g$ be a $\Modpq$-weak upper gradient of $f$. Then, for any
$\varepsilon\in(0,\fz)$, there exists a function $g_\varepsilon$, which is an
upper gradient of $f$, such that $\|g_\varepsilon-g\|_{\mpq}\leq\varepsilon$
and $g_\varepsilon\geq g$ everywhere on $\mscX$.
\end{lem}

Applying Lemma \ref{l-wug2}, we obtain the following conclusion.

\begin{thm}\label{t1}
Let $1<p\le q<\fz$. If $\mscX$
supports a weak $(1,p)$-Poincar\'{e} inequality
and the measure $\mu$ is doubling, then
$$\nmpq \hookrightarrow\hmpq.$$
\end{thm}

\begin{proof}
Let $f\in\nmpq$. By \cite[Theorem 1.0.1]{kz}, we see that
$\mscX$ supports a weak $(1,r)$-Poincar\'{e} inequality for some $r\in(1,p)$.
By Lemma \ref{l-wug2},
there exists an upper gradient $g$ of $u$ such that
$$\|f\|_{\mpq}+\|g\|_{\mpq}\ls \|f\|_{\nmpq}.$$
Since $\mscX$ supports a weak $(1,r)$-Poincar\'{e} inequality
for some $r\in(1,p)$, by \cite[Theorem 3.2]{hk00},
we know that there exists a set $E\subset \mscX$ with
$\mu(E)=0$ such that, for all $x,\,y\in \mscX\setminus E$,
\begin{equation*}
|f(x)-f(y)|\ls d(x,y)\lf\{\lf[M(g^r)(x)\r]^{1/r}+\lf[M(g^r)(y)\r]^{1/r}\r\}.
\end{equation*}
Hence a positive constant multiple of $h:=[M(g^r)]^{1/r}$ is a Haj\l asz gradient of $f$.
Then, by Lemma \ref{lem_bd of hl in morrey}, we know that
$$\|f\|_{\hmpq}\ls \|f\|_{\mpq}+\|h\|_{\mpq}
\ls \|f\|_{\mpq}+\|g\|_{\mpq}\ls \|f\|_{\nmpq},$$
which completes the proof of Theorem \ref{t1}.
\end{proof}

\begin{rem}
When $p=q$, under the same assumptions as in Theorem \ref{t1}, it was proved
by Shanmugalingam in \cite[Theorem 4.9]{Sh} that $NM^p_p(\mscX)=HM^p_p(\mscX)$
with equivalent norms.
\end{rem}

Next we turn to consider the inverse embedding of Theorem \ref{t1}.

\begin{thm}\label{t2}
Let $1\le p\le q<\fz$. Then,
$$\hmpq\hookrightarrow\nmpq.$$
\end{thm}

\begin{proof}
Let $f\in \hmpq$. Then, there exists a Haj\l asz gradient $h\in \cm^q_p(\mscX)$ of $f$
such that
\begin{equation*}
|f(x)-f(y)|\le d(x,y) [h(x)+h(y)],\quad x,\,y\in \mscX\setminus E,
\end{equation*}
for some $E$ of measure $0$, and
$\|f\|_{\cm^q_p(\mscX)}+\|h\|_{\cm^q_p(\mscX)}\ls \|f\|_{\hmpq}.$
It was proved in \cite[Theorem 1.1]{jyy} that, if $f,\,g\in L^1_\loc(\mscX)$ and
$g$ is a Haj\l asz gradient of $f$, then there exist $\wz f$ and $\wz g$
such that $\wz f=f$ and $\wz g=g$ almost everywhere, and $8\wz g$ is an upper
gradient of $\wz f$. Since $\wz f=f$ in $\hmpq$, we  identify
$f$ and $\wz f$. In this sense, $8\wz h$ is an upper
gradient of $f$.
Therefore,
$$
\|f\|_{\nmpq}\ls \|f\|_{\cm^p_q(\mscX)}+\lf\|\wz h\r\|_{\cm^q_p(\mscX)}
\sim\|f\|_{\cm^q_p(\mscX)}+\|h\|_{\cm^q_p(\mscX)}\ls \|f\|_{\hmpq},
$$
which completes the proof of Theorem \ref{t2}.
\end{proof}

Combining Theorems \ref{t1} and \ref{thm-h=n}, we have the following conclusion.

\begin{thm}\label{thm-h=n}
Let $1<p\le q<\fz$. If $\mscX$
supports a weak $(1,p)$-Poincar\'{e} inequality
and the measure $\mu$ is doubling, then $\nmpq=\hmpq$ with equivalent norms.
\end{thm}

Next we consider the relations among the Haj{\l}asz-Morrey-Sobolev
space, the Newton-Morrey-Sobolev
space and the classical Morrey-Sobolev space on $\rn$.
Let $1\le p\le q<\fz$.
Recall that the classical
\emph{Morrey-Sobolev space} $\wmpq$ is defined by
$$WM^q_p(\rn):=\lf\{f\in \mathcal{M}^q_p(\rn):\
|\nabla f|\in \mathcal{M}^q_p(\rn)\r\},$$
where $\nabla f$ denotes the \emph{weak derivative} of $f$. The
\textit{norm} of $f\in
WM^q_p(\rn)$ is given by
$$\|f\|_{WM^q_p(\rn)}:=\|f\|_{\mathcal{M}^q_p(\rn)}
+\||\nabla f|\|_{\mathcal{M}^q_p(\rn)}.$$
Observe that  $WM^p_p(\rn)$ is just the Sobolev space $W^{1,p}(\rn)$.

\begin{thm}\label{t3}
Let $1<p\le q<\fz$. Then,
$$WM^q_p(\rn)=NM^q_p(\rn)=HM^q_p(\rn)$$
with equivalent norms.
\end{thm}

\begin{proof}
Observe that the conclusion of Theorem \ref{t3} when $1<p=q<\fz$
is just \cite[Theorem 1]{ha}. Thus, in what follows of this proof,
we always assume that $1<p<q<\fz$.

By Theorem \ref{thm-h=n}, it suffices to prove that
$WM^q_p(\rn)\hookrightarrow HM^q_p(\rn)$ and
$$NM^q_p(\rn)\hookrightarrow WM^q_p(\rn).$$
We first show that $WM^q_p(\rn)\hookrightarrow HM^q_p(\rn)$. Let
$f\in WM^q_p(\rn)$. By the definition of $\wmpq$, we see that
$|\nabla f|\in L^p(Q)$ for all cubes $Q$ in $\rn$ and
then, following the argument as in \cite[p.\,404]{ha},
we know that, for all Lebesgue points $x,\,y\in\rn$ of $f$,
\begin{eqnarray*}
|f(x)-f(y)|\ls |x-y|\lf[M(|\nabla f|)(x)+M(|\nabla f|)(y)\r].
\end{eqnarray*}
Hence a positive constant multiple of $M(|\nabla f|)$ is
a Haj\l asz gradient of $f$. Moreover, by
Definition \ref{def hmpu} and Lemma \ref{lem_bd of hl in morrey},
we further see that
\begin{eqnarray*}
\|f\|_{HM^q_p(\rn)}&\le& \|f\|_{\mathcal{M}^q_p(\rn)}
+\|M(|\nabla f|)\|_{\mathcal{M}^q_p(\rn)}\\
&\ls& \|f\|_{\mathcal{M}^q_p(\rn)}
+\||\nabla f|\|_{\mathcal{M}^q_p(\rn)}\approx \|f\|_{WM^q_p(\rn)}.
\end{eqnarray*}
This shows that $WM^q_p(\rn)\subset HM^q_p(\rn)$.

Next we prove $NM^q_p(\rn)\subset WM^q_p(\rn)$.
Let $f\in NM^q_p(\rn)$. Then, by Definition \ref{def nm space},
there exists $g\in\cm_p^q(\rn)$ such that $g$ is a weak upper
gradient of $f$. Moreover, for any ball $B\subset\rn$,
we have $g\in L^p(B)$, which implies that $f\in N^{1,p}(B).$
By \cite[Theorem A.2]{bb}, we know that, for $i\in\{1,\ldots,n\}$
and almost every $x\in B$, $\frac{\partial f}{\partial x_i}(x)$
exists and $|\frac{\partial f}{\partial x_i}(x)|$ is
controlled by $g(x)$. Since $B\subset\rn$ is arbitrary,
it follows that, for almost every $x\in\rn$,
$$\lf|\frac{\partial f}{\partial x_i}(x)\r|\le g(x),$$
which, together with $g\in\cm_p^q(\rn)$, implies that
$\frac{\partial f}{\partial x_i}\in\cm_p^q(\rn)$. Furthermore,
we have $|\nabla f|\in\cm_p^q(\rn)$, from which,
together with $f\in\cm_p^q(\rn)$, we deduce that
$f\in\wmpq$. This finishes the proof of Theorem \ref{t3}.
\end{proof}

We remark that Theorem \ref{t3} when $p=q$ goes back to the equivalence
between Sobolev spaces and Haj\l asz-Sobolev spaces on $\rn$ obtained in \cite{ha}.

\begin{rem}\label{rd}
(i) We remark that, \emph{for all $1< p<q<\fz$, the set
$C^1(\rn)$ of functions having continuous derivatives up
to order $1$ is not dense in $NM^q_p(\rn)$}.
To see this, by Theorem \ref{t3}, we only need to consider $WM^q_p(\rn)$.
For simplicity, we only consider the case $n=1$.
Let $\phi\in C_c^\fz(\rr)$ such that $0\le \phi\le1$, $\phi\equiv 1$ on $(-1,1)$,
and $\phi\equiv 0$ on $(-2,2)^c$. Write $g(x):=|x|^{1-1/q}\phi(x)$ for all $x\in\rr$.
Then, the function
\begin{equation*}
g'(x):=\left\{\begin{array}{lll}
(1-1/q)x^{-1/q}\phi(x)+x^{1-1/q}\phi'(x),\quad &x\in(0,\fz);\\
0,\quad &x=0;\\
-(1-1/q)(-x)^{-1/q}\phi(x)+(-x)^{1-1/q}\phi'(x),\quad &x\in(-\fz,0)
                  \end{array}
\right.
\end{equation*}
is a weak derivative of $g$.
Since it is known that
$|x|^{\az}\chi_{(-2,2)}(x)\in \cm_p^q(\rr)$ if and only if $\az\ge -1/q$,
we then see that $g\in WM_p^q(\rr)$.

Now we apply an approach from \cite[pp. 587-588]{z86} to show that
$g$ can not be approximated by $C^1(\rr)$ functions in $WM_p^q(\rr)$.
Indeed, it suffices to prove that $g'$ can not be
approximated by continuous functions in $\cm_p^q(\rr)$.
To see this, for any continuous function $h$,
write $N:=\sup_{x\in (-1,1)}|h(x)|^p<\fz$. Notice that $\phi\equiv 1$ on $(-1,1)$.
We then see that, for all $R\in(0,1)$,
 \begin{eqnarray*}
\int_{-R}^R \lf|g'(x)-h(x)\r|^p\,dx
&&\ge 2^{-p}\int_{-R}^R\lf|g'(x)\r|^p\,dx-2NR.
\end{eqnarray*}
Notice that
\begin{eqnarray*}
\int_{-R}^R\lf|g'(x)\r|^p\,dx
&&\ge (1-1/q)^p \int_0^R x^{-p/q}\,dx=\frac{(1-1/q)^p}{1-p/q}R^{1-p/q}.
\end{eqnarray*}
We know that \begin{eqnarray*}
\int_{-R}^R \lf|g'(x)-h(x)\r|^p\,dx
&&\ge 2^{-p}\frac{(1-1/q)^p}{1-p/q}R^{1-p/q}-2NR\\
&&=R^{1-p/q}\lf[2^{-p}\frac{(1-1/q)^p}{1-p/q}-2N R^{p/q}\r].
\end{eqnarray*}
Hence, taking $R$ small enough such that
$$2^{-p}\frac{(1-1/q)^p}{1-p/q}-2N R^{p/q}\ge2^{-p-1}\frac{(1-1/q)^p}{1-p/q},$$
 we then see that
\begin{eqnarray*}
\lf\|g'-h\r\|_{\cm_p^q(\rr)}^p&&\gs R^{p/q-1} \int_{-R}^R \lf|g'(x)-h(x)\r|^p\,dx\\
&&\gs 2^{-p-1}\frac{(1-1/q)^p}{1-p/q}>0.
\end{eqnarray*}
This implies the above claim.

(ii) We point out that the key property we used in (i) is the locally
boundedness of continuous functions, which ensures that the number $N$ is finite.
If we replace continuous functions $h$ by any locally bounded functions, then
the subsequent argument remains true.
From this observation, together with the well-known fact that
any Lipschitz function $f$
on $\rn$ is differentiable almost everywhere and the absolute
value $|\partial_i f|$ of its weak derivative $\partial_i f$ is
dominated by its Lipschitz constant $L_f$ almost everywhere,
we deduce that $g$ can not be approximated by any  Lipschitz function $f$
in the norm of  $WM_p^q(\rn)$.
Therefore, the \emph{set of Lipschitz functions is not dense in
$NM_p^q(\rn)$ when $1<p<q<\fz$}.
\end{rem}

\section{Boundedness of (fractional) maximal operators \label{s6}}

\hskip\parindent  This section is devoted to the boundedness of
(fractional) maximal operators on Morrey type spaces over
metric measure spaces.

In Subsection \ref{s6.1}, for a
geometrically doubling metric measure space $(\mscX,d,\mu)$
in the sense of Hyt\"onen \cite{h10}, we show, in Theorem \ref{t6.5} below, that the
modified maximal operator $M_0^{(\beta)}$ (see \eqref{eq-def-mfmo}
below) is bounded on the modified Morrey
space  $\cm_p^{q,(k)}(\mscX)$, which, when $(\mscX,d,\mu):=
(\rr^n,|\cdot|,\mu)$ with $\mu$ being a Radon measure
satisfying the polynomial growth condition (also called the
 non-doubling measure), was introduced by Sawano and Tanaka
\cite{st05}. As an application,  the boundedness of
the fractional maximal operator $M_\alpha^{(\beta)}$
on this space is also obtained in Proposition \ref{pf} below.

In Subsection \ref{s6.2}, if $\mu$ is a doubling measure,
as applications of Theorem \ref{t6.5} and Proposition \ref{pf},
we show the boundedness of the fractional
maximal operator $M_\alpha$ on Morrey
spaces (see Corollary \ref{t6.1} below),
from which, we further deduce, in Corollary \ref{c6.9} below,
the boundedness of the fractional maximal operator $\wz M_\alpha$
on Morrey spaces when $\mu$ further satisfies the measure lower bound condition
(see \eqref{Q-regular} below). If $\mu$ is doubling,
satisfies \eqref{Q-regular} and has the relative $1$-annular decay property
(see \eqref{d-adp} below), we then obtain
the boundedness of $\wz M_\az$ on $\hmpq$ (see Theorem \ref{tf} below).
Finally, we prove that, if $\mu$ is doubling and satisfies \eqref{Q-regular},
and $\mscX$ supports a weak $(1,p)$-Poincar\'e
inequality, then the discrete fractional maximal function $M_\alpha^\ast$
is bounded on $\nmpq$ (see Theorem \ref{thm-frac-max-bd1} below).

\subsection{Maximal operators on $\cm_p^{q,(k)}(\mscX)$\label{s6.1}}

\hskip\parindent
In 2010, Hyt\"onen \cite{h10} introduced the notion of
geometrically doubling metric measure spaces which
include both spaces of homogeneous type and
the Euclidean spaces with non-doubling measures satisfying
the polynomial growth condition as special cases; see also the monograph \cite{yyh13}
for some recent developments of this subject.

Now we recall the following notion of the geometrically doubling from \cite{h10},
which is also known as \emph{metrically doubling} (see, for example, \cite[p.\,81]{he}).

\begin{defn}\label{def gdp}
A metric space $(\mscX,d)$ is said to be
\emph{geometrically doubling}, if there exists $N_0\in\nn$ such that
any given ball contains no more than $N_0$ points at distance exceeding
half its radius.
\end{defn}

From the geometrically doubling property, we deduce
the following conclusion, which
is used later on.

\begin{prop}\label{prop cover}
Let $(\mscX,d)$ be a geometrically doubling metric space.
Then, for any ball $B(x,r)\subset\mscX$,
with $x\in\mscX$ and $r\in (0,\fz)$, and any $n_1\ge n_2>1$,
there exist $r_0\in(0,\fz)$  and $\wz{N}$ balls
$\{B(x_i,r_0)\}_{i=1}^{\wz{N}}$ such that $n_1B(x_i,r_0)\subset n_2B(x,r)$
for all $i\in\{1,\ldots,\wz N\}$ and
$$B(x,r)\subset\bigcup_{i=1}^{\wz N}B(x_i,r_0),$$
where $\wz{N}\in\nn$ depends only on $n_1,\ n_2$ and
the constant $N_0$ in Definition \ref{def gdp}.
\end{prop}

\begin{proof}
Let $n_1$ and $n_2$ be as in Proposition \ref{prop cover}, and
$$k:=\lf\lfloor\log_2\frac{n_1+1} {n_2-1}\r\rfloor+1,$$
where $\lfloor t\rfloor$ denotes the
\emph{maximal integer not more than $t\in\rr$}. We claim that, for any $y\in\mscX$
and ball $B(x,r)\subset\mscX$ with $x\in\mscX$ and $r\in (0,\fz)$, if $B(y,\frac r{2^k})\cap
B(x,r)\ne\emptyset$, then $n_1B(y,\frac r{2^k})
\subset n_2B(x,r)$. Indeed, by choosing $z\in B(y,\frac r{2^k})\cap B(x,r)$ and
observing that $k>\log_2\frac{n_1+1}{n_2-1}$, we have
$$d(x,y)\le d(x,z)+d(z,y)
<\lf(1+\frac 1{2^k}\r)r<\lf(n_2-\frac{n_1}{2^k}\r)r.$$
Thus, for all
$w\in n_1B(y,\frac r{2^k})$,
$$d(w,x)\le d(w,y)+d(y,x)<\frac{n_1r}{2^k}+
\lf(n_2-\frac{n_1}{2^k}\r)r=n_2r,$$
which shows the above claim. Then, by repeating
the proof that (1) implies (2) in
\cite[Lemma 2.3]{h10}, we obtain the desired conclusion, which completes
the proof of Proposition \ref{prop cover}.
\end{proof}

Now we recall the definition of the modified Morrey space, which,
when $(\mscX,d,\mu):=
(\rr^n,|\cdot|,\mu)$ with $\mu$ being a Radon measure
satisfying the polynomial growth condition,
was originally introduced by
Sawano and Tanaka \cite{st05}.

\begin{defn}\label{def frac-morrey}
Let $k\in(0,\fz)$, $1\le p\le q<\fz$ and $\mscX$ be  a
metric measure space. The \emph{modified Morrey space}
$\cm_p^{q,(k)}(\mscX)$ is defined as
\begin{equation*}
\cm_p^{q,(k)}(\mscX):=\lf\{f\in L_\loc^{p}(\mscX):\ \|f\|_{\cm_p^{q,(k)}(\mscX)}<\fz\r\},
\end{equation*}
where
\begin{equation*}
\|f\|_{\cm_p^{q,(k)}(\mscX)}:=\sup_{B(x,r)\subset\mscX}[\mu(B(x,kr))]^{1/q-1/p}
\lf[\int_{B(x,r)}|f(y)|^p\,d\mu(y)\r]^{1/p},
\end{equation*}
where the supremum is taken over all balls $B(x,r)$, with $x\in\mscX$ and $r\in (0,\fz)$, of $\mscX$.
\end{defn}

Recall that a \emph{geometrically doubling metric measure space $(\mscX,d,\mu)$}
means that $(\mscX,d)$ is geometrically doubling and $\mu$ is a non-negative Radon measure
on $(\mscX,d)$.

\begin{prop}\label{equiv}
Let $(\mscX,d,\mu)$ be a geometrically doubling metric measure space
and $1\le p\le q<\fz$. Then, the space $\cm_p^{q,(k)}(\mscX)$ is
independent of the choice of $k\in(1,\fz)$.
\end{prop}

\begin{proof}
Let $k_1$, $k_2\in(1,\fz)$. We need to show that $\cm_p^{q,(k_1)}(\mscX)$
and $\cm_p^{q,(k_2)}(\mscX)$ coincide with equivalent norms.
To this end, without loss of generality, we may assume that $k_1<k_2$.
By Definition \ref{def frac-morrey},
we easily find that $\cm_p^{q,(k_1)}(\mscX)\subset\cm_p^{q,(k_2)}(\mscX).$
Thus, we still need to show the inverse embedding. Let $B$ be a ball
in $\mscX$. By Proposition \ref{prop cover}, there exist $\wz N$ balls
$\{B_i\}_{i=1}^{\wz N}$ with the same radius such that,
for all $i\in\{1,\ldots,\wz N\}$,
$k_2B_i\subset k_1B$ and $B\subset\cup_{i=1}^{\wz N}B_i$, where
$\wz N$ depends only on $k_1$, $k_2$ and $N_0$ in
Definition \ref{def gdp}. By these, we see that
\begin{eqnarray*}
[\mu(k_1B)]^{1/q-1/p}\lf[\int_B|f(x)|^p\,d\mu(x)\r]^{1/p}
&&\le\sum_{i=1}^{\wz N}[\mu(k_1B)]^{1/q-1/p}\lf[\int_{B_i}|f(x)|^p\,d\mu(x)\r]^{1/p}\\
&&\le\sum_{i=1}^{\wz N}[\mu(k_2B_i)]^{1/q-1/p}\lf[\int_{B_i}|f(x)|^p\,d\mu(x)\r]^{1/p}\\
&&\le\wz N\|f\|_{\cm_p^{q,(k_2)}(\mscX)}.
\end{eqnarray*}
By the arbitrariness of $B$ and Definition \ref{def frac-morrey}, we
conclude that $$\|f\|_{\cm_p^{q,(k_1)}(\mscX)}\le\wz{N}\|f\|_{\cm_p^{q,(k_2)}(\mscX)},$$
which further implies that $\cm_p^{q,(k_2)}(\mscX)\subset\cm_p^{q,(k_1)}(\mscX)$ and hence
completes the proof of Proposition \ref{equiv}.
\end{proof}

Recall that, for $\az\in[0,1]$ and $\beta\in[1,\fz)$, the \emph{modified fractional
maximal operator} $M^{(\beta)}_\az$ is defined by setting, for all
$f\in L_\loc^1(\mscX)$ and $x\in\mscX$,
\begin{equation}\label{eq-def-mfmo}
M^{(\beta)}_\az f(x):=\sup_{r>0}{[\mu(B(x,\beta r))]^{\alpha-1}}
\int_{B(x,r)}|f(y)|\,d\mu(y).
\end{equation}
In particular, we write $M_\az :=M^{(1)}_\az$.

The maximal operator $M^{(\beta)}_0$ where $\beta\in(1,\fz)$ is bounded on the modified
Morrey spaces. To prove this, we need the following technical lemma.

\begin{lem}\label{lem-cover}
Let $\beta\in(1,\fz)$ and $(\mscX,d)$ be a geometrically doubling metric
space. Suppose that $\cb:=\{B(x_\lz,r_\lz)\}_{\lz\in\Lambda}$ such that
$\sup_{\lambda\in\Lambda}r_\lambda<\fz$. Then, there exist $J_\bz\in\nn$, depending only on
$\bz$ and $N_0$ in Definition \ref{def gdp}, and sub-families
of balls of $\cb$, $\cb_i:=\{B(x_\lz,r_\lz)\}_{\lambda\in\Lambda_i}$ with
$i\in\{1,\ldots,J_\bz\}$, such that
\begin{enumerate}
\item[\rm(i)] for each $i\in\{1,\ldots,J_\bz\}$, $\cb_i$ consists of
disjoint balls;
\item[\rm(ii)] for any $\lambda\in\Lambda$, there exists $\lambda'\in
\cup_{i=1}^{J_\bz}\Lambda_i$ such that $B(x_\lambda,r_\lambda)\subset
B(x_{\lambda'},\beta r_{\lambda'}).$
\end{enumerate}
\end{lem}

\begin{proof}
Let $R:=\sup_{\lambda\in\Lambda}r_\lambda$ and, for all
$j\in\zz_+$,
\begin{equation}
\ca_j:=\lf\{B(x_\lambda,r_\lambda):\ \lambda\in\Lambda,\
R\sqrt{\beta^{-j-1}}< r_\lambda\le R\sqrt{\beta^{-j}}\r\}.
\end{equation}
Here, we need $\bz\in(1,\fz)$ and, otherwise, $A_j=\emptyset$
for all $j\in\zz_+$. Let $\cd_0\subset\ca_0$ be a maximal subset in $\ca_0$ such that, for
any two distinct balls $B(x_{\lambda},r_{\lambda})$ and
$B(x_{\lambda'},r_{\lambda'})$ in $\cd_0$, it holds true that
$d(x_{\lambda},x_{\lambda'})>R(\sqrt{\beta}-1)$. By
such a choice, we know that, for any ball $B(x_\lambda,
r_\lambda)\in\ca_0$, there exists $B(x_{\lambda'},r_{\lambda'})
\in\cd_0$ such that $d(x_\lambda, x_{\lambda'})\le R(\sqrt\beta-1)$.
Furthermore, from $r_\lambda\le R$ and $\beta r_{\lambda'}> R
\sqrt\beta$, it follows that
$$B(x_\lambda,r_\lambda)\subset
B(x_{\lambda'},\beta r_{\lambda'}).$$

Let $\ce_0$ be the collection of balls $B(x_\lambda,r_\lambda)$
which belong to $\cb$ and satisfy that, for some $B(x_{\lambda'},r_{\lambda'})
\in\cd_0$, $B(x_\lambda,r_\lambda)\subset B(x_{\lambda'},\beta r_{\lambda'})$.
Obviously, $\ca_0\subset \ce_0.$

Let $m\ge1$. We now define $\cd_m$ and $\ce_m$
recursively. Suppose that $\cd_j$ and $\ce_j$
for $j\in\{0,\ldots,m-1\}$ has  already  been defined. Let $\cd_m\subset\ca_m\setminus
\cup_{j=0}^{m-1}\ce_j$ be a maximal subset satisfying that,
for all distinct balls $B(x_\lambda,r_\lambda)$ and $B(x_{\lambda'},
r_{\lambda'})$ in $\cd_m$, it holds true that
$$d(x_\lambda,x_{\lambda'})
>R\sqrt{\beta^{-j}}(\sqrt\beta-1).$$
Let $\ce_m$ be the collection
of balls $B(x_\lambda,r_\lambda)$ which belong to $\cb$ and
satisfy that, for some $B(x_{\lambda'},r_{\lambda'})\in\cd_m$,
$B(x_\lambda,r_\lambda)\subset B(x_{\lambda'},\beta r_{\lambda'})$.
Notice that, for any ball $B\in\ca_m$, we have either $B\in\ce_j$
for some $j\in\{0,\ldots,m-1\}$ or
$B\in\ca_m\setminus\cup_{j=1}^{m-1}\ce_{j}$. In the first case,
we can find a ball $B(x_{\lambda'},r_{\lambda'})\in\cd_j$ such that
$B\subset B(x_{\lambda'},\beta r_{\lambda'})$. In the second case, we can
find $B(x_{\lambda'},r_{\lambda'})\in\cd_m$ such that $B\subset
B(x_{\lambda'},\beta r_{\lambda'})$.

Due to the geometrically doubling condition, we can partition each
$\cd_j$ into disjoint sub-families, $\cd_{j,1},\ldots,\cd_{j,L_\beta}$,
where $L_\beta$ is a positive constant depending only on $\beta$ and the
geometrically doubling constant $N_0$, since, for
any $j\in\zz_+$ and any $B:=B(x_\lambda,r_\lambda)\in\cd_j$,
there are at most $L_\beta$ balls in $\cd_j$ intersect $B$. Indeed,
let $\cf_j$ be the collection of balls $B':=B(x_{\lambda'},r_{\lambda'})$
which belong to $\cd_j$ and intersect $B$. Let $y\in B'$ and $z\in B\cap B'$.
Then,
$$d(y,x_\lambda)\le d(y,x_{\lambda'})+d(x_{\lambda'},z)+d(z,x_\lambda)
<2r_{\lambda'}+r_\lambda\le3R\sqrt{\beta^{-j}}.$$
Thus, $B'\subset
B(x_\lambda,3R\sqrt{\beta^{-j}}).$ On the other hand,
by the choice of $\cd_j$, we know that
$$d(x_\lambda,x_{\lambda'})
>R\sqrt{\beta^{-j}}(\sqrt\beta-1)$$
and hence
$$B\lf(x_{\lambda}, \frac{\sqrt{\bz}-1}2r_{\lambda}\r)\bigcap
B\lf(x_{\lambda'},\frac{\sqrt{\bz}-1}2r_{\lambda'}\r)=\emptyset.$$
By the geometrically doubling property
of $\mscX$ and \cite[Lemma 2.3]{h10}, we see that there exists a constant
$L_\beta$, depending on $N_0$ and $\beta$, such that $\cf_j$ has no more than
$L_\beta$ balls.

Let $N_\beta\in\nn$ satisfy
\begin{equation}\label{5.3}
1+2\sqrt{\beta^{-N_\beta}}<\sqrt{\beta}.
\end{equation}
We claim that, if $j_1\ge j_2+N_\beta$, then, for any pair of balls,
$(B_1,B_2)\in\cd_{j_1}\times\cd_{j_2}$, $B_1$ and $B_2$ do not intersect.
To see this, assume that $B_1\cap B_2\neq \emptyset$ and $x\in B_1\cap B_2$.
Then, by \eqref{5.3}, for any $y\in B_1$, we have
\begin{eqnarray*}
d(y,x_2)&&\le d(y,x)+d(x,x_2)
<2r(B_1)+r(B_2)\le 2R\sqrt{\beta^{-j_1}}+R\sqrt{\beta^{-j_2}}\\
&&\le R\sqrt{\beta^{-j_2}}(2\sqrt{\beta^{-N_\bz}}+1)\le
R\sqrt{\beta^{-j_2-1}}\beta<\beta r(B_2),
\end{eqnarray*}
where $x_i$ and $r(B_i)$ denote the center and
the radius of $B_i$, for $i\in\{1,2\}$, respectively.
Thus, $B_1\subset \beta B_2$ and hence belongs to $\ce_{j_2}$,
which contradicts to
the definition of $\cd_{j_1}$, since $\cd_{j_1}\cap \ce_{j_2}=\emptyset$.
Thus, the above claim holds true.

Therefore, if, for $i\in\{1,\ldots,N_\bz\}$ and $n\in\{1,\ldots,L_\bz\},$ let
$$\cb_{i,n}:=\bigcup_{j=0}^\fz\cd_{N_\beta j+i,n}.$$
Then, $\{B_{i,n}:\ i\in\{1,\ldots,N_\bz\},\ n\in\{1,\ldots,L_\bz\}\}$
are the desired families, which completes the proof of Lemma \ref{lem-cover}.
\end{proof}

The boundedness of the modified maximal operator on $L^p(\mscX)$
could be deduced from the above lemma by borrowing some ideas
used in the proof of \cite[Section 3.1, Theorem 1]{s93}. We give
some details as follows.

\begin{thm}\label{thm-cmk-b>1}
Let $\beta\in(1,\fz)$ and $(\mscX,d,\mu)$ be a
geometrically doubling metric measure
space.

${\rm(i)}$ Then, there exists a positive constant $C$ such that,
 for all $\lambda\in(0,\fz)$ and $f\in L^1(\mscX)$,
 \begin{equation}
\mu\lf(\lf\{x\in\mscX:\ M_0^{(\beta)}(x)>\lambda\r\}\r)\le \frac{C}{\lambda}\|f\|_{L^1(\mscX)}.
\end{equation}

${\rm(ii)}$ Let $p\in(1,\fz]$. Then,
there exists a positive constant $C$ such that, for all $f\in L^p(\mscX)$,
\begin{equation}
\lf\|M_0^{(\beta)}f\r\|_{L^p(\mscX)}\le C\|f\|_{L^p(\mscX)}.
\end{equation}
\end{thm}

\begin{proof}
The boundedness of $M_0^{(\beta)}$ on $L^\fz(\mscX)$ is obvious.
Next we only prove (i), since (ii)
can be deduced from (i) and the $L^\fz(\mscX)$-boundedness of $M_0^{\bz}$ via
interpolation.

Let
\begin{equation*}
E_\lambda:=\lf\{x\in \mscX:\ M_0^{(\beta)}f(x)>\lambda\r\}.
\end{equation*}
Then, by the definition of $E_\lambda$, for any $x\in E$,
there exists a ball $B_x$ such that
\begin{equation}
\frac1{\mu(\beta B_x)}\int_{B_x}|f(y)|\,d\mu(y)>\lambda.
\end{equation}
For all $k\in \zz_+$,  let $\cb^{(k)}$
be the collection of all
balls $B_x$ for $x\in E$, whose radius $r(B_x)\in(0,2^k]$, and
$E_\lambda^{(k)}:=\{x\in E_\lambda:\ B_x\in \cb^{(k)}\}$.
Then,  $E_\lambda=\cup_{k\in\zz_+}E_\lambda^{(k)}$ and
$E_\lambda^{(k)}\subset E_\lambda^{(k+1)}$ for any $k\in\zz_+$.

For each $k\in\zz_+$, by Lemma \ref{lem-cover}, we can find $J_\beta\in\nn$, independent of $k$, and
sub-families $\cb^{(k)}_i\subset\cb^{(k)},\ i\in\{1,\ldots,J_\beta\}$ such that
$$ \bigcup_{B\in \cb^{(k)}} B \subset \bigcup_{i=1}^{J_\beta}
\bigcup_{B\in\cb^{(k)}_i}\beta B,$$
where $\beta B$ denotes the ball with the same center as $B$
but $\beta$ times the radius of $B$. Thus, by the fact that
$E_\lambda^{(k)}$ increasingly converges to $E_\lambda$ as  $k\to \fz$,
and the disjointness of balls in $\cb_i^{(k)}$ over $i$,
we see that
\begin{eqnarray*}
\mu(E_\lz)&&=\lim_{k\to\fz}\mu\lf(E_\lz^{(k)}\r)\le \lim_{k\to\fz}
\mu\lf(\bigcup_{B\in \cb^{(k)}} B\r)\le \lim_{k\to\fz}\mu\lf(\bigcup_{i=1}^{L_\beta}
\bigcup_{B\in\cb^{(k)}_i}\beta B\r)\\
&&\le \lim_{k\to\fz}\sum_{i=1}^{L_\beta}\sum_{B\in\cb_i^{(k)}}
\mu(\beta B)
\le\lim_{k\to\fz}\sum_{i=1}^{L_\beta}\sum_{B\in\cb_i^{(k)}}
\frac1\lambda\int_B|f(y)|\,d\mu(y)
\ls\frac{L_\beta}\lambda
\|f\|_{L^1(\mscX)}.
\end{eqnarray*}
This finishes the proof of Theorem \ref{thm-cmk-b>1}.
\end{proof}

\begin{rem}\label{rem-mf}
(i) It is worth pointing out that Theorem \ref{thm-cmk-b>1} also holds true
for the non-centered maximal operator, whose proof is similar, the details
being omitted.

(ii) We should point out that Lemma \ref{lem-cover} and Theorem
\ref{thm-cmk-b>1} are generously provided to us
by Professor \textbf{Yoshihiro Sawano} from Tokyo Metropolitan University
of Japan.

(iii) Lemma \ref{lem-cover} and Theorem
\ref{thm-cmk-b>1} in the case $\bz=1$ are still unknown.
\end{rem}

Then we have the following conclusion, which generalizes \cite[Theorem 2.3]{st05},
wherein the corresponding result on the non-doubling measure satisfying the
polynomial growth condition on $\rn$ was obtained. The proof of Theorem
\ref{t6.5} is similar to that of \cite[Theorem 2.3]{st05},
and one key tool used in the proof is the
$L^p(\mu)$-boundedness in  Theorem \ref{thm-cmk-b>1}.
For the sake of convenience, we give the details.

\begin{thm}\label{t6.5}
Let $\mscX$ be a geometrically doubling metric measure space,
$1<p\le q<\fz$, $\beta\in(1,\fz)$ and $k\in(1,\fz)$.
Then, there exists a positive constant $C$ such that, for all $f\in\cm_{p}^{q,(k)}(\mscX)$,
\begin{equation}\label{eq-6.1}
\lf\|M^{(\beta)}_0f\r\|_{\cm_p^{q,(k)}(\mscX)}\le C\|f\|_{\cm_p^{q,(k)}(\mscX)}.
\end{equation}
\end{thm}

\begin{proof}
By Proposition \ref{equiv}, it suffices to consider the case
that $k:=\frac{2\beta}{\beta+1}>1$.
Let  $f\in\cm_p^{q,(k)}(\mscX)$ and $B_0\subset\mscX$ be a ball.
Define $\wz\beta:=\frac{\beta+7}{\beta-1}>1$, $f_1:=f\chi_{\wz\beta B_0}$ and $f_2:=f-f_1$.
Then, by Definition \ref{def frac-morrey}, together with
$1<p\le q< \fz$, we have
\begin{eqnarray*}
&&\frac1{[\mu(k\wz\beta B_0)]^{1/p-1/q}}
\lf\{\int_{B_0}\lf[M^{(\beta)}_0f_1(y)\r]^p\,d\mu(y)\r\}^{1/p}\\
&&\hs\le\frac1{[\mu(k\wz\beta B_0)]^{1/p-1/q}}
\lf\{\int_{\mscX}\lf[M^{(\beta)}_0f_1(y)\r]^p\,d\mu(y)\r\}^{1/p}\\
&&\hs\ls\frac1{[\mu(k\wz\beta B_0)]^{1/p-1/q}}
\lf\{\int_{\wz\beta B_0}|f(y)|^p\,d\mu(y)\r\}^{1/p}\ls\|f\|_{\cm_p^{q,(k)}(\mscX)},
\end{eqnarray*}
where we used the fact that $M^{(\beta)}_0$ is bounded on
$L^p(\mu)$, for $p\in(1,\fz)$ (see Theorem \ref{thm-cmk-b>1}).

To estimate $f_2$, observe that, if $B\subset\mscX$ is a ball satisfying that
$B\cap B_0\ne\emptyset$ and $B\cap(\mscX\setminus\wz\beta B_0)\ne\emptyset$, then
the radius $r_B>\frac{\wz\beta-1}2r_{B_0}=\frac{4}{\beta-1}r_{B_0}$,
where $r_B$ and $r_{B_0}$ denote, respectively, the radii of $B$ and $B_0$, and hence
$B_0\subset \frac{\beta+1}2B$.
Therefore, we see that, for any $x\in B_0$,
\begin{eqnarray*}
M^{(\beta)}_0f_2(x)&&\le\sup_{x\in B}\frac1{\mu(\beta B)}\int_B|f_2(y)|\,d\mu(y)
\le\sup_{B_0\subset \frac{\beta+1}2 B}\frac1{\mu(\beta B)}\int_B|f(y)|\,d\mu(y)\\
&&\le \sup_{B_0\subset B}\frac1{\mu(\frac{2\beta}{\beta+1} B)}\int_B|f(y)|\,d\mu(y)=
\sup_{B_0\subset B}\frac1{\mu(k B)}\int_B|f(y)|\,d\mu(y),
\end{eqnarray*}
and hence, by this, the H\"older inequality and Definition
\ref{def frac-morrey}, we further see that
\begin{eqnarray*}
&&\frac1{[\mu(k\wz\beta B_0)]^{1/p-1/q}}
\lf[\int_{B_0}\lf[M^{(\beta)}_0f_2(y)\r]^p\,d\mu(y)\r]^{1/p}\\
&&\hs\le\frac{[\mu(B_0)]^{1/p}}{[\mu(k\wz\beta B_0)]^{1/p-1/q}}\sup_{B_0\subset B}
\frac1{\mu(k B)}\int_B|f(y)|\,d\mu(y)\\
&&\hs\le\frac{[\mu(B_0)]^{1/p}}{[\mu(k\wz\beta B_0)]^{1/p-1/q}}\sup_{B_0\subset B}
\frac{[\mu(B)]^{1-1/p}}{\mu(kB)}\lf[\int_B|f(y)|^p\,d\mu(y)\r]^{1/p}\\
&&\hs\le\sup_{B_0\subset B}\frac{[\mu(B_0)]^{1/p}}{[\mu(k\wz\beta B_0)]^{1/p-1/q}}
\frac{[\mu(B)]^{1-1/p}}{\mu(kB)}[\mu(k B)]^{1/p-1/q}\|f\|_{\cm_p^{q,(k)}(\mscX)}\\
&&\hs\le\sup_{B_0\subset B}\frac{[\mu(B)]^{1-1/p+1/q}}{[\mu(kB)]^{1-1/p+1/q}}\|f\|_{\cm_p^{q,(k)}(\mscX)}
\le\|f\|_{\cm_p^{q,(k)}(\mscX)},
\end{eqnarray*}
where the last inequality follows from the fact that $B_0\subset B$ and $k,\ \wz\beta>1$.
This estimate for $f_2$, together with the previous estimate for $f_1$ and Proposition \ref{equiv},
further implies that
\begin{equation*}
\lf\|M^{(\beta)}_0f\r\|_{\cm_p^{q,(k)}(\mscX)}\approx
\lf\|M^{(\beta)}_0f\r\|_{\cm_p^{q,(k\wz\beta)}(\mscX)}\ls\|f\|_{\cm_p^{q,(k)}(\mscX)},
\end{equation*}
which completes the proof of Theorem \ref{t6.5}.
\end{proof}

\begin{rem}
It is still unknown whether the conclusions of Theorems \ref{thm-cmk-b>1}
 and \ref{t6.5} hold true or not when $\beta=1$. Indeed, it is known that
 $M^{(1)}_0$ might not be bounded on $L^p(\mscX)$ with $p\in(1,\fz)$
when $\mu$ is not doubling; see, for example, \cite{ntv}.
\end{rem}

Using Theorem \ref{t6.5}, we further have the following boundedness of
the modified fractional maximal operator $M^{(\beta)}_\alpha$
on the modified Morrey space.

\begin{prop}\label{pf}
Let $\mscX$ be a geometrically doubling metric measure space,
$1<p\le q<\fz$, $\beta\in(1,\fz)$, $\alpha\in(0,1/q)$ and $k\in(1,\fz)$.
Then, there exists a positive constant $C$ such that, for all $f\in\cm_{{p}}^{{q},(k)}(\mscX)$,
\begin{equation*}
\lf\|M^{(\beta)}_\alpha f\r\|_{\cm_{\wz p}^{\wz q,(k)}(\mscX)}
\le C\|f\|_{\cm_{p}^{q,(k)}(\mscX)},
\end{equation*}
 where $\wz p:=\frac{p}{1-\alpha q}$ and
$\wz q:=\frac{q}{1-\alpha q}$.
\end{prop}

\begin{proof}
By Proposition \ref{equiv},
it suffices to consider the case that $k=\beta\in(1,\fz)$.

For any ball $B(x,r)\subset\mscX$, with $x\in\mscX$ and $r>0$,
and $f\in\cm_{{p}}^{{q},(\beta)}(\mscX)$,
we know, by the H\"older inequality, \eqref{a} and \eqref{def max}, that
\begin{eqnarray*}
&&[\mu(B(x,\beta r))]^{{\alpha}-1}\int_{B(x,r)}|f(y)|\,d\mu(y)\\
&&\hs=[\mu(B(x,\beta r))]^{{\alpha}-1}\lf[\int_{B(x,r)}|f(y)|\,d\mu(y)\r]^{{\alpha q}}
\lf[\int_{B(x,r)}|f(y)|\,d\mu(y)\r]^{1-{\alpha q}}\\
&&\hs\le[\mu(B(x, \beta r))]^{{\alpha}-1+(1-\frac1p){\alpha q}}
\lf[\int_{B(x,r)}|f(y)|^p\,d\mu(y)\r]^{\frac{\alpha q}{p}}
\lf[\int_{B(x,r)}|f(y)|\,d\mu(y)\r]^{1-{\alpha q}}\\
&&\hs=\lf\{\frac{1}{[\mu(B(x, \beta r))]^{1-\frac{p}q}}\int_{B(x,r)}|f(y)|^p
\,d\mu(y)\r\}^{\frac{\az q}{p}}
\lf[\frac1{\mu(B(x,\beta r))}\int_{B(x,r)}|f(y)|\,d\mu(y)\r]^{1-{\alpha q}}\\
&&\hs\le\|f\|_{\cm_{{p}}^{{q},(\beta)}(\mscX)}^{{\az q}}\lf[M_0^{(\beta)}f(x)\r]^{1-{\alpha q}},
\end{eqnarray*}
which, together with  \eqref{eq-def-mfmo}, implies that, for all $x\in\mscX$,
\begin{equation*}
M^{(\beta)}_\az f(x)
\le\|f\|_{\cm_{{p}}^{{q},(\beta)}(\mscX)}^{\alpha q}\lf[M_0^{(\beta)}f(x)\r]^{1-\alpha q}.
\end{equation*}
Then, by Theorem \ref{t6.5}, we see that
\begin{equation*}
\lf\|M^{(\beta)}_\az f\r\|_{\cm_{\wz p}^{\wz q,(\beta)}(\mscX)}
\le\|f\|_{\cm_{{p}}^{{q},(\beta)}(\mscX)}^{{\alpha q}}
\lf\|\lf[M_0^{(\beta)}f\r]^{1-{\alpha q}}\r\|_{\cm_{\wz p}^{\wz q,(\beta)}(\mscX)}
\ls \|f\|_{\cm_{{p}}^{{q},(\beta)}(\mscX)},
\end{equation*}
which completes the proof of Proposition \ref{pf}.
\end{proof}

\subsection{Fractional maximal operators
on $\hmpq$ and $\nmpq$ \label{s6.2}}

\hskip\parindent Recently, Heikkinen et al. \cite{hknt,hlnt} studied
the boundedness of some (fractional) maximal operators on the Newton-Sobolev space
and the Haj\l asz-Sobolev space over metric measure spaces.
In this section, we consider the corresponding problem
for Newton-Morrey-Sobolev spaces and Haj\l asz-Morrey-Sobolev spaces.

Throughout this section, we \emph{always assume that the measure $\mu$ is doubling}.
We call a measure is doubling if there exists a constant $C_0\in[1,\fz)$ such that,
for all $x\in\mscX$ and $r>0$,
\begin{equation}\label{db}
\mu(B(x,2r))\le C_0\mu(B(x,r))\quad \textup{(doubling property)}.
\end{equation}

Recall that it is well known that any space of homogeneous type
is also geometrically doubling (see \cite[pp.\,66-67]{cw71}).
Moreover, since $\mu$ is doubling, we see that, for any $\beta\in (1,\fz)$
and $\az\in[0,1]$, there exists a positive constant $C$, depending on $\bz$ and $\az$, such that,
for all $f\in L^1_\loc(\mscX)$ and $x\in\mscX$,
$M_\alpha f(x)\le C{M}^{(\beta)}_\alpha f (x)$.
From these facts, Theorem \ref{t6.5} and Proposition \ref{pf},
we immediately deduce the following conclusion.

\begin{cor}\label{t6.1}
Let $1<p\leq q<\fz$ and $\alpha\in[0,1/q)$.
Then, there exists a positive constant $C$, depending on $\az$, $p$ and $q$,
such that, for all $f\in\mpq$,
\begin{equation*}
\|{M}_\alpha f\|_{\cm_{\widetilde{p}}^{\widetilde{q}}(\mscX)}\leq C\|f\|_{\mpq},
\end{equation*}
 where $\widetilde{p}:=\frac{p}{1-\alpha q}$ and
$\widetilde{q}:=\frac{q}{1-\alpha q}.$
\end{cor}

As an application of Corollary \ref{t6.1},
we obtain the following boundedness of (fractional) maximal operators
$\wz M_\alpha$ on modified Morrey spaces.

Recall that a measure $\mu$ is said to satisfy the
\emph{measure lower bound condition}, if
there exists a positive constant $C$ such that, for any $x\in\mscX$ and $r\in(0,\fz)$,
\begin{equation}\label{Q-regular}
\mu(B(x,r))\ge Cr^Q
\end{equation}
for some $Q\in(0,\fz)$.

Recently, if $\mu$ satisfies \eqref{Q-regular}, Heikkinen et al. \cite{hlnt}
established the boundedness from
$L^p(\mscX)$ to $L^s(\mscX)$ for $p\in (1, Q)$ and $s:=\frac{Qp}{Q-\alpha p}$
of the following modified (fractional) maximal function $\wz M_\alpha$,
defined by setting, for any $\alpha\in[0,1]$,
$f\in L_{\loc}^1(\mscX)$ and $x\in\mscX$,
\begin{equation}
\wz M_\alpha f(x):=\sup_{r>0}\frac{r^\alpha}{\mu(B(x,r))}\int_{B(x,r)}|f(y)|\,d\mu(y).
\end{equation}
It is easy to see that, in the present setting,
there exists a positive constant $C$, depending
on $\az$ and $Q$, such that, for all $f\in L_\loc^1(\mscX)$
and $x\in\mscX$,
$\wz M_\alpha f(x)\le C M_\az f(x)$,
which, together with Corollary \ref{t6.1}, implies
the following conclusion.

\begin{cor}\label{c6.9}
Let $1<p\leq q<\fz$ and $\alpha\in[0,1/q)$. Assume that $\mu$ satisfies \eqref{Q-regular}.
Then, there exists a positive constant $C$, depending on $\az$, $p$ and $q$,
such that, for all $f\in\mpq$,
\begin{equation}\label{e6.2}
\lf\|{\wz M}_\alpha f\r\|_{\cm_{\widetilde{p}}^{\widetilde{q}}(\mscX)}\leq C\|f\|_{\mpq},
\end{equation}
 where $\widetilde{p}:=\frac{p}{1-\alpha q}$ and
$\widetilde{q}:=\frac{q}{1-\alpha q}.$
\end{cor}

Recall that $\mscX$
is said to satisfy the \emph{relative $1$-annular decay property}, if there 	
exists a positive constant $C$ such that, for all $x\in\mscX$, $R\in(0,\fz)$ and
$h\in(0,R)$,
\begin{equation}\label{d-adp}
\mu\lf(B\cap \lf[B(x,R)\setminus B(x, R-h)\r]\r)\le C \frac h{r_B} \mu(B)
\end{equation}
for all balls $B$ with radius $r_B<3R$; see, for example, \cite[(2.5)]{hlnt}.

Now we turn to the boundedness of the (fractional) maximal operator ${\wz M}_\alpha$ on
Haj\l asz-Morrey-Sobolev spaces.

\begin{thm}\label{tf}
Assume that $\mu$ satisfies \eqref{Q-regular} and $\mscX$
has the relative 1-annular decay property \eqref{d-adp}. Let $1<p\le q<\fz$
and $\alpha\in[0,Q/q)$. Then, for any $f\in\hmpq$, $\wz M_\alpha f\in\hmpsqs$,
 where $p^\ast:=\frac{Qp}{Q-\alpha q}$ and
$q^\ast:=\frac{Qq}{Q-\alpha q}$. Moreover, there exists a positive constant $C$,
depending only on the doubling constant, $Q$, $p$,
$q$ and $\alpha$, such that, for all $f\in\hmpq$,
\begin{equation*}
\lf\|\wz M_\alpha f\r\|_{\hmpsqs}\leq C\|f\|_{\hmpq}.
\end{equation*}
\end{thm}

\begin{proof}
The proof is similar to that of \cite[Theorem 4.5]{hlnt}.
We present some details.
Let $f\in\hmpq$ and $g\in\mpq$ be a Haj\l asz gradient of $f$ such that
$\|g\|_{\mpq}\ls \|f\|_{\hmpq}$.
It is easy to see that $g$ is also a Haj\l asz gradient of $|f|$.
Let $r\in(1,p)$ and define
\begin{equation*}
\wz g:=\lf[\wz M_{\alpha r}(g^r)\r]^{1/r}.
\end{equation*}
By an argument similar to that used in the proof of \cite[Theorem 4.5]{hlnt}, we know
that $\wz g$ is a Haj\l asz gradient of $\wz M_\alpha (|f|)$, as well as $\wz M_\alpha f$, since
$\wz M_\alpha (|f|)=\wz M_\alpha f$.
Moreover, by $p/r>1$, \eqref{a} and \eqref{e6.2}, we see that
\begin{equation}\label{6.11}
\|\wz g\|_{\mpsqs}=\lf\|\wz M_{\alpha r}(g^r)\r\|_{\cm_{p^\ast/r}^{q^\ast/r}(\mscX)}^{1/r}
\lesssim\|g^r\|_{\cm_{p/r}^{q/r}(\mscX)}^{1/r}\approx\|g\|_{\mpq}.
\end{equation}
Combining \eqref{6.11} and Definition \ref{def hmpu},
we obtain the desired conclusion and then
complete the proof of Theorem \ref{tf}.
\end{proof}

We point out that Theorem \ref{tf} when $p=q$ goes back to
\cite[Theorem 4.5]{hlnt}.

Now we recall the \emph{discrete (fractional) maximal operator} $M_\alpha^\ast$
introduced in \cite[Section 5]{hknt}. Let $\{B(x_i,r)\}_{i\in\nn}$ be a ball covering of $\mscX$
such that $\{B(x_i,r)\}_{i\in\nn}$ are of finite overlap. Since $\mscX$ is doubling,
the overlap number $N$ depends only on the doubling constant and
is independent of $r$. Let $\{\vz_i\}_{i\in\nn}$  be
a partition of unity related to $\{B(x_i,r)\}_{i\in\nn}$ such that $0\le\vz_i\le1$, $\vz_i=0$ on $\mscX\setminus
B(x_i,6r)$, $\vz_i\ge v$ on $B(x_i,3r)$ and $\vz_i$ is Lipschitz function with Lipschitz constant $L/r$,
where $L\in(0,\fz)$ and $v\in(0,1]$ are constants depending only on the doubling constant,
and $\sum_{i\in\nn}\vz_i\equiv1$. The discrete convolution of $u\in L^1_{\rm loc}(\mscX)$
at the scale $3r$ is defined by setting, for all $x\in\mscX$,
$$u_r(x):=\sum_{i\in\nn}\vz_i(x) u_{B(x_i,3r)},$$
where $u_{B(x_i,3r)}$ denotes the integral mean of $u$ on $B(x_i,3r)$ (see \eqref{eq-def-mean}).
Now, let $\{r_j\}_{j\in\nn}$ be a sequence of the positive rational numbers,
and $\{B(x_{i,j},r_j)\}_{i\in\nn}$ for each $j$ is a ball covering of $\mscX$ as above.
Then, the  \emph{discrete (fractional) maximal function} $M_\az^\ast u$ of $u$ is defined
by setting, for all $x\in\mscX$,
$$M_\alpha^\ast u(x):= \sup_{j\in\nn} r_j^\az |u|_{r_j}(x).$$
Similar to the proof of
\cite[Theorem 6.3]{hknt}, we obtain the following result on the boundedness of
$M_\alpha^\ast$ on Newton-Morrey-Sobolev spaces.

\begin{thm}\label{thm-frac-max-bd1}
Let $\mu$ satisfy \eqref{Q-regular}, $1<p\le q<\fz$ and $\alpha\in[0,Q/q)$.
Assume that $\mscX$ is complete and supports a weak $(1,p)$-Poincar\'e inequality.
Then, for any $f\in\nmpq$, it holds true that $M_\alpha^\ast f\in NM_{p^\ast}^{q^\ast}
(\mscX)$  with $p^\ast:=Qp/(Q-\alpha q)$ and $q^\ast:=Qq/(Q-\alpha q)$.
Moreover, there exists a positive constant $C$, independent of $f$, such that
\begin{equation*}
\|M_\alpha^\ast f\|_{NM_{p^\ast}^{q^\ast}(\mscX)}\leq C\|f\|_{\nmpq}.
\end{equation*}
\end{thm}

\begin{proof}
Let $f\in\nmpq$ and $g\in\mpq$ be a $\modpq$-weak upper gradient of $f$ such that
\begin{equation}\label{g}
\|g\|_{\mpq}\le 2\|f\|_{\nmpq}.
\end{equation}
By \cite[Lemma 5.1]{hknt} and \eqref{e6.2}, we have
\begin{equation}\label{h}
\|M_\alpha^\ast f\|_{\cm_{p^\ast}^{q^\ast}(\mscX)}\ls\|f\|_{\mpq}.
\end{equation}
By the same reason as that used in the
proof \cite[Theorem 6.3]{hknt}, observing that the pointwise Lipschitz
constant of a function is also an upper gradient of that function,
we see that a positive constant multiple of
$(M_{\alpha \theta}^\ast g^{\theta})^{1/{\theta}}$ is a $\modpq$-weak upper gradient of
$M_\alpha^\ast f$, where $\theta$ lies in $(1,p)$ such that the weak
$(1,\theta)$-Poincar\'e inequality is supported by $\mscX$. By $g^{\tz}\in\cm_{p/\theta
}^{q/\theta}(\mscX)$ and $p/\theta>1$, together with \cite[Lemma 5.1]{hknt}
and \eqref{e6.2}, we know that
\begin{equation}\label{i}
\lf\|\lf[M_{\alpha \theta}^\ast(g^{\theta})\r]^{1/\theta}\r\|_{\cm_{p^\ast}^{q^\ast}(\mscX)}
\ls\|g\|_{\mpq}.
\end{equation}
Combining \eqref{g}, \eqref{h} and \eqref{i}, we obtain
\begin{eqnarray*}
\|M_\alpha^\ast f\|_{NM_{p^\ast}^{q^\ast}(\mscX)}
&&\ls\|M_\alpha^\ast f\|_{\cm_{p^\ast}^{q^\ast}(\mscX)}
+\lf\|\lf[M_{\alpha \theta}^\ast(g^{\theta})\r]^{1/\theta}\r\|_{\cm_{p^\ast}^{q^\ast}(\mscX)}\\
&&\ls\|f\|_{\mpq}+\|g\|_{\mpq}\ls\|f\|_{\nmpq},
\end{eqnarray*}
which completes the proof of Theorem \ref{thm-frac-max-bd1}.
\end{proof}

We remark that Theorem \ref{thm-frac-max-bd1} when $p=q$ goes back to
\cite[Theorem 6.3]{hknt}.

\vspace{0.3cm}

\noindent{\bf Acknowledgements.} The authors would like to deeply thank
Professor Pekka Koskela, Professor Nagewari Shanmugalingam
and Dr. Renjin Jiang
for some helpful discussions on the subject of this article, and
Dr. Luk\'a\v{s} Mal\'y for delivering
a copy of his Ph.\,D thesis to us.
The authors would also like to express their deep thanks  to
Professor Yoshihiro Sawano, who generously provides us
the proofs of Lemma \ref{lem-cover} and Theorem \ref{thm-cmk-b>1} of this article,
and to the referee for his/her carefully reading and so many helpful and useful comments
which essentially improves this article.

\bigskip

\noindent Yufeng Lu and Dachun Yang (Corresponding author)

\medskip

\noindent  School of Mathematical Sciences, Beijing Normal University,
Laboratory of Mathematics and Complex Systems, Ministry of
Education, Beijing 100875, People's Republic of China

\smallskip

\noindent {\it E-mails}: \texttt{yufeng.lu@mail.bnu.edu.cn} (Y. Lu)

\hspace{1.14cm}\texttt{dcyang@bnu.edu.cn} (D. Yang)

\bigskip

\noindent  Wen Yuan

\medskip

\noindent  School of Mathematical Sciences, Beijing Normal University,
Laboratory of Mathematics and Complex Systems, Ministry of
Education, Beijing 100875, People's Republic of China

and

\noindent  Mathematisches Institut, Friedrich-Schiller-Universit\"at Jena,
Jena 07743, Germany

\smallskip

\noindent{\it E-mail:} \texttt{wenyuan@bnu.edu.cn}
\end{document}